\documentclass[11pt,leqno]{amsart}

\usepackage[all,cmtip]{xy}
\usepackage{a4,latexsym}
\usepackage[utf8]{inputenc}
\usepackage{textcomp,fontenc,exscale,ifthen}
\usepackage[hidelinks]{hyperref}
\usepackage{amsxtra,amsmath,amsfonts,amscd, amssymb, mathrsfs, amsthm}
\usepackage{color, mathtools}
\usepackage{enumitem}

\newtheorem{thm}{Theorem}[section]
\newtheorem{prop}[thm]{Proposition}
\newtheorem{cor}[thm]{Corollary}
\newtheorem{lem}[thm]{Lemma}
\newtheorem{lemma}[thm]{Lemma}

\theoremstyle{definition}
\newtheorem{definition}[thm]{Definition}

\newtheorem{rem}[thm]{Remark}
\newtheorem{art}[thm]{}

\numberwithin{paragraph}{section}
\numberwithin{equation}{thm}

\setlist[enumerate]{label=\it{(\roman*)},
ref=\it{(\roman*)}}

\usepackage{color}
\usepackage{comment}

\newcommand{\walter}[1]{{\textcolor{black}{#1}}}

\def\N{{\mathbb N}}
\def\Z{{\mathbb Z}}
\def\Q{{\mathbb Q}}
\def\R{{\mathbb R}}

\def\A{{\mathbb A}}

\def\L{{\mathbb L}}
\def\N{{\mathbb N}}
\def\T{{\mathbb T}}

\newcommand{\metr}{{\|\hspace{1ex}\|}}
\newcommand{\Hom}{{\rm Hom}}

\def\an{{\rm an}}

\DeclareMathOperator{\trop}{trop}
\DeclareMathOperator{\tropbar}{\overline{trop}}
\DeclareMathOperator{\MA}{MA}

\newcommand{\Xan}{{X^{\rm an}}}

\newcommand{\Tan}{{T^{\rm an}}}

\newcommand{\Acal}{{\mathscr A}}
\newcommand{\Bcal}{{\mathscr B}}
\newcommand{\Ccal}{{\mathscr C}}
\newcommand{\Dcal}{{\mathscr D}}
\newcommand{\Ecal}{{\mathscr E}}

\newcommand{\Hcal}{{\mathscr H}}

\newcommand{\Ncal}{{\mathscr N}}
\newcommand{\Ocal}{{\mathscr O}}

\newcommand{\codim}{{\rm codim}}

\newcommand{\Spec}{{\rm Spec}}
\newcommand{\Spf}{{\rm Spf}}

\DeclareMathOperator{\supp}{supp}

\newcommand{\relint}{{\rm relint}}
\newcommand{\kcirc}{{ K^\circ}}
\newcommand{\ktilde}{{ \tilde{K}}}
\newcommand{\Ufrak}{{\mathfrak U}}
\newcommand{\Vfrak}{{\mathfrak V}}
\newcommand{\Xfrak}{{\mathfrak X}}
\newcommand{\Afrak}{{\mathfrak A}}
\newcommand{\Bfrak}{{\mathfrak B}}
\newcommand{\Efrak}{{\mathfrak E}}
\newcommand{\Hfrak}{{\mathfrak H}}
\newcommand{\Lfrak}{{\mathfrak L}}
\newcommand{\Lan}{{L^{\rm an}}}

\newcommand{\qgamma}{{(\Q,\Gamma)}}

\DeclareMathOperator{\red}{{red}}
\DeclareMathOperator{\PL}{{PL}}
\DeclareMathOperator{\Sk}{{Sk}}

\title[Monge--Amp\`ere measures for toric metrics on abelian varieties]{Monge--Amp\`ere measures for toric metrics on abelian varieties}

\author[W.~Gubler]{Walter Gubler}
\address{W. Gubler, Mathematik, Universit{\"a}t 
Regensburg, 93040 Regensburg, Germany}
\email{walter.gubler@mathematik.uni-regensburg.de}

 \author[S.~Stadl\"oder]{Stefan Stadl\"oder}
 \address{S. Stadl\"oder,  Mathematik, Universit{\"a}t 
Regensburg, 93040 Regensburg, Germany}
\email{stefan.stadloeder@mathematik.uni-regensburg.de}

 \thanks{W.~Gubler 
was supported by the collaborative research 
center SFB 1085 \emph{Higher Invariants - Interactions between Arithmetic Geometry and Global Analysis} funded by the Deutsche Forschungsgemeinschaft. S.~Stadl\"oder was supported by the Hanns-Seidel-Stiftung and the  Studienstiftung des deutschen Volkes.}

\begin{document}

\begin{abstract}
Toric metrics on a line bundle of an abelian variety $A$ are the invariant metrics under the natural torus action coming from Raynaud's uniformization theory. We compute here the associated Monge--Amp\`ere measures for the restriction to any closed subvariety of $A$. This generalizes the computation of canonical measures done by the first author from canonical metrics to toric metrics and from discrete valuations to arbitrary non-archimedean fields.
\end{abstract}

\keywords{Berkovich analytic spaces, formal geometry, abelian varieties, canonical measures} 
\subjclass{{Primary 14G40; Secondary 11G10, 14G22}}

\maketitle

\setcounter{tocdepth}{1}

\tableofcontents

\section{Introduction}  \label{section: introduction}

Abelian varieties are projective geometrically integral group varieties over a field. They play a distinguished role in arithmetic geometry. Let $X$ be a closed subvariety of an abelian variety $A$ over a number field $K$. The group structure of $A$ makes it easier to understand the structure of the $K$-rational points of $X$. For example, Faltings \cite{faltings91} showed the Bombieri--Lang conjecture for such $X$. No finiteness statements are sensible for $\overline K$-rational points of $X$, instead we are looking for density statements for special points. The Manin--Mumford conjecture, proven by Raynaud \cite{raynaud83}, states that the set of torsion points of $X$ is dense if and only if $X$ is the translate of an abelian subvariety by a torsion point. 

The height of a $\overline K$-rational point of a projective variety measures the arithmetic complexity of its coordinates. In the case of an abelian variety $A$, there are canonical heights called N\'eron--Tate heights. A natural generalization of the Manin--Mumford conjecture is the Bogomolov conjecture which claims that if the closed subvariety $X$ of $A$ has dense small points, then $X$ is again a torsion translate of an abelian subvariety. This was shown by Ullmo \cite{ullmo98} for a curve inside its Jacobian and by Zhang \cite{zhang98} in full generality. All the above statements have analogues in the case of a function field $K$ where one has to take into account that  constant abelian varieties are also a source for points of height $0$. The function field variant of the Bogomolov conjecture is called  the geometric Bogomolov conjecture which  was harder to prove than the number field case. It was shown by Gao and Habegger \cite{gao_habegger19} in the case of the function field of a curve and generalized by Cantat--Gao--Habegger--Xie \cite{cantat_etal21} to arbitrary function fields, but both assuming that $K$ has characteristic $0$. In arbitrary characteristics, the geometric Bogomolov conjecture was shown  by Xie and Yuan \cite{xie_yuan} 
using reduction steps by Yamaki 
and the Manin--Mumford conjecture over function fields by Hrushowski \cite{hrushovski01} and Pink--Roessler \cite{pink_roessler04}.

Ullmo's and Zhang's argument relies on an equidistribution theorem for small points due originally to Szpiro--Ullmo--Zhang \cite{szpiro_ullmo_zhang}, later generalized by Yuan \cite{yuan-2008}. This equidistribution strategy also works to some extent  in the case of function fields as shown for totally degenerate abelian varieties in \cite{gubler-2007b}. In constrast to the number field case, the equidistribution has then to be with respect to a non-archimedean place and takes place on the associated Berkovich space. The argument relies on a precise description of canonical measures of $X$, see below for more details.  This description holds for all abelian varieties $A$ \cite{gubler-compositio} and was the key in Yamaki's argument showing that it is enough to prove the geometric Bogomolov conjecture for abelian varieties $A$ with good reduction at all places of $K$. In the present paper, we will generalize the description of canonical measures of $X$.

For the remainder of the introduction, we consider an algebraically closed field $K$ endowed with a complete non-archimedean absolute value and non-trivial value group $\Gamma$ in $\R$. For a projective variety $X$, we will perform analytic considerations on the associated Berkovich space $X^\an$. The notion of continuous semipositive metrics of a line bundle $L$ over $X$ goes back to Zhang and is recalled in \S \ref{subsection fomospm}. For such a metric $\metr$, Chambert--Loir \cite{chambert-loir-2006} has introduced  non-archimedean Monge--Amp\`ere measures $c_1(L,\metr)^{\wedge \dim(X)}$ which are positive Radon measures on $\Xan$, see \cite{chambert-loir-2006}, \cite{gubler-2007b} and \S \ref{subsection nonarMA}.

Assume now that $X$ is a closed subvariety of an abelian variety $A$ over $K$ and let $d \coloneqq \dim(X)$. For a rigidified ample line bundle $L$ of $A$, there is a canonical metric $\metr_L$ of $L$. Since $\metr_L$ is a continuous semipositive metric, we get the \emph{canonical measure}
$$\mu_L \coloneqq c_1(L|_X,\metr_L)^{\wedge d}$$
on the Berkovich analytification $\Xan$ of $X$. If $X,A$ and $L$ are defined over a discretely valued field, then it was shown in \cite{gubler-compositio} that the support of $\mu_L$ has a piecewise linear structure with a polytopal decomposition $\Dcal$ such that 
$$\mu_L = \sum_{\sigma \in \Dcal}  r_\sigma \mu_\sigma$$
where $r_\sigma \in \R_{\geq 0}$ and $\mu_\sigma$ is a Lebesgue measure on the polytope $\sigma$. Note that lower dimensional polytopes are also allowed. The goal of this paper is to generalize these results, removing the discreteness assumption about the field of definition  and replacing canonical metrics by a more general class called toric metrics. As we will see, toric metrics on $L$ are the variations of canonical metrics by combinatorial means.

We continue with the above setup confirming that $K$ is any algebraically closed non-archimedean field with non-trivial absolute value. The Raynaud extension for the abelian variety $A$ is a canonical exact sequence
	$$0 \longrightarrow \Tan \longrightarrow E^\an \stackrel{q}{\longrightarrow} B^\an \longrightarrow 0$$
of abelian analytic groups over $K$ which are all algebraic with $T$ a  torus of rank $n$ and $B$ an abelian variety of good reduction. 
Raynaud's uniformization theory gives a canonical description $A^\an = E^\an/\Lambda$ where $E$ is a group scheme of finite type over $K$ and $\Lambda$ is a discrete subgroup of $E^\an$ contained in $E(K)$. Note that the quotient map $p\colon E^\an \to A^\an$ is in general not algebraic. Moreover, there is a \emph{canonical tropicalization}
$\trop \colon E^\an \rightarrow N_\R$
mapping $\Lambda$ homeomorphically onto a lattice of $N_\R \cong \R^n$ where  $N$ is the cocharacter lattice of $T$. It induces a canonical tropicalization
$$\tropbar\colon A^\an \longrightarrow N_\R/\trop(\Lambda).$$
We say that a continuous metric $\metr$ of the rigidified line bundle $L$ of $A$  is \emph{toric} if there is a function $\phi\colon N_\R \to \R$ such that  $p^*\metr = e^{-\phi\circ \trop}p^*\metr_L$.

There is a rigidified line bundle $H$ on $B$ such that we have an identification $p^*(L^\an)=q^*(H^\an)$ as $\Lambda$-linearized cubical line bundles on $E^\an$. The  metric $q^*(\metr_H)$  does not descend to $L^\an$ and the obstruction leads to 
 a cocycle $(z_\lambda)_{\lambda \in \trop(\Lambda)}$ encoding all tropical information about the line bundle $L$, see Section \ref{section tormet} for details.

\begin{thm} \label{theo-intro-toric-metrics}
There is a bijective correspondence between continuous toric metrics $\metr$ on $L^\an$ and continuous functions $f\colon N_\R \to \R$ satisfying the cocycle rule
$$f(\omega + \lambda) = f(\omega)+ z_\lambda(\omega) \quad (\omega \in N_\R\, , \lambda \in \trop(\Lambda)).$$
The correspondence is determined by 
$$f \circ \trop = -\log(p^*\metr/q^*\metr_H).$$
%If $f$ is convex, then  $\metr$ is semipositive. If $L$ is ample, the converse also holds. 
If $L$ is ample, then the function $f$ is convex if and only if the metric $\metr$ is semipositive. 
\end{thm}

This will be shown in Proposition \ref{bijective correspondence for toric metrics on A} and Theorem \ref{semipositive toric metrics}. To deduce semipositivity from convexity, we will use an approximation result by piecewise linear convex functions satisfying the cocycle rule which was done in \cite{burgos-gubler-jell-kuennemann2}. The converse uses arguments from the theory of weakly smooth forms on Berkovich analytic spaces given in \cite{gubler_rabinoff_jell:harmonic_trop} and recalled in Appendix \ref{section app diffforms}.

Recall that $X$ is a closed $d$-dimensional subvariety of the abelian variety $A$.
In Section \ref{section cansubset}, we show that for any ample line bundle $L$ on $A$, the support $S_X$ of the canonical measure $ c_1(L|_X,\metr_L)^{\wedge d}$ has a canonical piecewise $(\Q,\Gamma)$-linear structure  not depending on the choice of $L$. In fact, we will show that it is a $(\Q,\Gamma)$-skeleton in the sense of Ducros \cite{ducros12:squelettes_modeles}.

\begin{thm} \label{intro canonical subset and tropicalization}
	The canonical tropicalization map $\tropbar\colon A^\an \to N_\R/\trop(\Lambda)$ restricts to a piecewise $(\Q,\Gamma)$-linear map $S_X \to \tropbar(\Xan)$ which is surjective and finite-to-one.  
\end{thm}

This result was shown  in \cite{gubler-2007b} in the special case of $X,L,A$ being defined over a discretely valued field and was crucial in Yamaki's reduction step mentioned above. We prove at the end of the paper that this holds for any algebraically closed non-archimedean field $K$.

The main result of this paper describes the non-archimedean Monge--Amp\`ere measure of a continuous toric metric in terms of the classical real Monge--Amp\`ere measure  $\MA(f)$ associated to a convex function on $\R^n$, see \S \ref{subsection real MA}. 

%\begin{thm} \label{intro main result}
%There is a polytopal $\qgamma$-decomposition $\Sigma$ of the canonical subset $S_X$ such that for any  ample line bundle $L$ on $A$ with continuous toric metric $\metr$ corresponding to the convex function $f$ as in Theorem \ref{theo-intro-toric-metrics}, there is a multiplicity $m_\sigma \in \Q_{\geq 0}$ associated to  $\sigma \in \Sigma$ such that 
%$$c_1(L|_X,\metr)^{\wedge d}(\Omega)=  m_\sigma \cdot \MA(f)(\Omega)$$
%for any Lebesgue measurable subset $\Omega$ of $\relint(\sigma)$.
%\end{thm}

\begin{thm} \label{intro main result}
	There is a polytopal $\qgamma$-decomposition $\Sigma$ of the canonical subset $S_X$ such that for any  ample line bundle $L$ on $A$ with continuous toric metric $\metr$ corresponding to the convex function $f$ as in Theorem \ref{theo-intro-toric-metrics}, there is a multiplicity $m_\sigma \in \Q_{\geq 0}$ associated to  $\sigma \in \Sigma$ such that 
	$$c_1(L|_X,\metr)^{\wedge d}(\Omega)=  m_\sigma \cdot \MA(f)(\tropbar(\Omega))$$
	for any Lebesgue measurable subset $\Omega$ of $\relint(\sigma)$.
	The multiplicity $m_\sigma$ depends only on $X$, $L$ and $\sigma$, but not on the toric metric $\metr$.
\end{thm}

In the special case of the canonical metric, we can say more:

\begin{cor} \label{intro cor of main result}
The above polytopal decomposition $\Sigma$ has the property that for any ample line bundle $L$ of $A$, there is $r_\sigma \in \R_{\geq 0}$ associated to $\sigma \in \Sigma$ such that  
$$c_1(L|_X,\metr_L)^{\wedge d}= \sum_{\sigma \in \Sigma} r_\sigma \mu_\sigma$$
where $\mu_\sigma$ is a fixed choice of a Lebesgue measure on the polytope $\sigma \in \Sigma$.	
\end{cor}

We begin proving Theorem \ref{intro main result} by showing a variant (given in Theorem \ref{toric MA on X'}) for the pull-back to a strictly polystable alteration of $X$ where the support is contained in the union of the canonical faces of the skeleton which are non-degenerate with respect to the alteration. The existence of such a strictly polystable alteration follows from a result of Adiprasito, Liu, Pak and Temkin \cite{adiprasito_etal}. In Theorem \ref{piecewise linear structure of canonical subset}, we will see that the induced morphism from the union of these non-degenerate faces to $S_X$ is a piecewise $(\Q,\Gamma)$-linear surjective map  which is finite-to-one. Then Theorem \ref{intro main result} follows from the projection formula \eqref{projection formula}.

The structure of the paper is as follows. Section \ref{section notprel} fixes the notation and gives the preliminaries on convex geometry, non-archimedean geometry, formal models and semipositive metrics, and real and non-archimedean Monge--Amp\`ere measures. In Section \ref{section plappr}, we deal with piecewise linear convex approximations of convex functions in a purely combinatorial setting. The main result is Proposition \ref{transversal pl approximation} where we show that such an approximation is possible by preserving a cocycle rule. The approximations can be chosen such that the underlying domains of linearity are transversal to a given fixed set of polytopes. This will be crucial later. In Section \ref{section tormet}, we first recall Raynaud's uniformization theory. Then we 
introduce toric metrics and prove Theorem \ref{theo-intro-toric-metrics}. Finally, we recap the theory of formal Mumford models of an abelian variety $A$ over $K$. Mumford models have the advantage that they can be described in combinatorial terms on $\tropbar(A^\an)=N_\R/\trop(\Lambda)$. 

In Section \ref{section strictpolyalt}, we first recall strictly polystable alterations for a closed subvariety $X$ of $A$, the piecewise linear structure of the skeleton $\Sk(\Xfrak')$ of the underlying strictly polystable formal scheme $\Xfrak'$ over $\kcirc$ and that any polytopal decomposition of $\Sk(\Xfrak')$ leads to a formal model $\Xfrak''$ of the generic fiber of $\Xfrak'$ which dominates $\Xfrak'$. 
Then we relate this construction to the formal Mumford models of $A$ and  give a combinatorial formula for the  degree of an irreducible component of the special fiber of $\Xfrak''$ under a transversality assumption. All the material from Section \ref{section strictpolyalt} is a direct generalization of \cite[\S 5]{gubler-compositio} from the strictly semistable to the strictly polystable case. In Section \ref{section MA tormetr}, we prove the variant of Theorem \ref{intro main result} on the strictly polystable alteration. We use  the piecewise linear approximation from Proposition \ref{transversal pl approximation} to reduce to the piecewise linear case and then the claim is a direct consequence of the combinatorial degree formula from Section \ref{section strictpolyalt}. Finally, in Section \ref{section cansubset}, we prove the claims about the canonical subset.

\subsection*{Acknowledgements}
We thank Antoine Ducros for a fruitful discussion about skeletons of Berkovich spaces and we are grateful to Felix Herrmann for proofreading the text linguistically. We thank Jos\'e Burgos, Roberto Gualdi, Klaus K\"unnemann and Joe Rabinoff for comments to an earlier version of this paper. We are grateful to the referee for the careful reading and the valuable suggestions helping to improve the presentation.

\section{Notation and preliminaries}  \label{section notprel}

\subsection{Basic conventions} \label{subsection: notation} 
\addtocontents{toc}{\protect\setcounter{tocdepth}{1}}

The set of natural numbers $\N$ includes $0$.  A \emph{lattice} in a finite dimensional real vector space is a discrete subgroup which generates the vector space. For an abelian group $M$ and a subgroup $G$ of $\R$,  we set $M_G \coloneqq
M \otimes_\Z G$. 
By a compact space, we mean a quasi-compact Hausdorff space.

A ring is always assumed to be commutative and with $1$. 
The group of invertible elements in a ring $A$ is denoted by $A^\times$.
 A \emph{variety} over a field $F$ is an integral scheme which is of finite type and separated over $\Spec\, F$. 
 
By Bourbaki's approach to measure theory, a \emph{positive Radon measure} on a locally compact Hausdorff space $X$ can be seen as a positive linear functional on the space of compactly supported continuous real functions $C_c(X)$ of $X$. By the Riesz represention theorem, such a function is given by $f \to \int_X f(x) \, d\mu(x)$ for a unique regular Borel measure $\mu$ on $X$. A sequence of Radon measures $\mu_k$ is called \emph{weakly convergent} to a Radon measure $\mu$ on $X$ if 
$$\lim_k \int_X f(x) \, d\mu_k(x) = \int_X f(x) \, d\mu(x)$$
for all $f \in C_c(X)$.

\subsection{Convex geometry} \label{subsection congeom}

Let $N$ be a free abelian group of rank $n$ and $M=\Hom_\Z(N,\Z)$ its dual. %We denote by  $N_\R$ resp.~$M_\R$ the respective scalar extensions to $\R$. 
A function $f\colon N_\R \to \R$ is called \emph{affine} if $f=u+c$ for some $u \in M_\R$ and $c \in \R$. 
Then $u$ is called the \emph{slope} of $f$.  
%and \emph{integral slope} means $u \in M$.  
For a subring $A$ of $\R$ and a $A$-submodule $\Gamma$ of $\R$, we say that $f$ is \emph{$(A,\Gamma)$-affine} if $u \in M_A$ and $c \in \Gamma$. 

A finite intersection of half-spaces $\{f \leq 0\}$ for affine functions $f$ on $N_\R$ is called a \emph{polyhedron in $N_\R$}. It is called a \emph{$(A,\Gamma)$-polyhedron} if  the  affine functions $f$ can be chosen $(A,\Gamma)$-affine. A \emph{polytope} is a bounded polyhedron. 
The \emph{relative interior} of a polyhedron  $\sigma$ is denoted by $\relint(\sigma)$. 
For a polyhedron $\sigma$, a \emph{face} is  the intersection of $\sigma$ with the boundary of a half-space containing $\sigma$. 
By convention, we allow $\sigma$ and $\emptyset$ also as faces of $\sigma$. 
The notation  $\tau\prec\sigma$ means that $\tau$ is a face of  $\sigma$. 
A \emph{polyhedral complex} in $N_\R$ is a locally finite set $\Ccal$ of polyhedra in $N_\R$ such that for $\sigma\in \Ccal$, the faces of $\sigma$ are in $\Ccal$ and for  $\sigma,\rho \in \Ccal$ we have that $\sigma \cap \rho$ is a common face of $\sigma$ and $\rho$. 
The \emph{support} of $\Ccal$ is defined by  $|\Ccal| \coloneqq \bigcup_{\sigma \in \Ccal} \sigma$. For $k \in \N$, we set $\Ccal_k \coloneqq \{\sigma \in \Ccal \mid \dim(\sigma)=k\}$. 
 A function $f\colon C \to \R$ on a closed subset $C$ of $N_\R$ is called \emph{piecewise linear} if there is a polyhedral complex $\Ccal$ with support $C$ such that $f|_\sigma$ is affine for all $\sigma \in \Ccal$. If we can choose $\Ccal$ as a $(A,\Gamma)$-polyhedral complex (i.e.~ a polyhedral complex consisting of $(A,\Gamma)$-polyhedra) such that all $f|_\sigma$ are $(A,\Gamma)$-affine functions, then we call $f$  \emph{piecewise $(A,\Gamma)$-linear}.

More generally, a \emph{piecewise $(A,\Gamma)$-linear space} is a locally compact Hausdorff space $X$  with a compact atlas $(X_i)_{i \in I}$ by charts to $(A,\Gamma)$-polytopes in $\R^{n_i}$ such that the transition functions are (piecewise) $(A,\Gamma)$-affine and such that every point has a neighbourhood in $X$ given by a finite union of $X_i$'s. All the above notions are transferred to $X$ by using the polytopal charts. We refer to \cite[\S 0]{ducros12:squelettes_modeles} for more details.

\subsection{Real Monge--Amp\`ere measures} \label{subsection real MA}

Let $N$ be a free abelian group of rank $n$ with dual $M$ and let $f\colon \Omega \to \R$ be a convex function on an open convex subset $\Omega$ of $N_\R$. Then a  classical construction from real analysis gives the \emph{Monge--Amp\`ere measure} $\MA(f)$ which is a positive Radon measure on $\Omega$. Let $\lambda_N$ be the Haar measure on $N_\R$ normalized by requiring that the covolume of the lattice $N$ is one. 
For $f \in C^2(\Omega)$, we have
$$\MA(f) = n!\det\bigl((\partial_{ij}f)_{1 \leq i,j \leq n}\bigr)\lambda_N,\qquad \partial_{ij}f=\frac{\partial^{2} f}{\partial u_{i}\partial u_{j}},$$ 
where $u_1,\dots, u_n$ is a basis for $M$ viewed as coordinates on $N_\R$. For any convex function $f$, the construction of the \emph{Monge--Amp\`ere measure} $\MA(f)$ is local with respect to the open convex set $\Omega$ in $N_\R$ and continuous with respect to uniform convergence of convex functions and weak convergence of Radon measures. 

What we need is that for a conic piecewise linear function $f$ on $N_\R$ centered at $x\in N_\R$ (i.e.~ $f(r(\omega-x)+x)-f(x)=r(f(\omega)-f(x))$ for all $\omega \in N_\R$ and $r>0$), the measure $\MA(f)$ is the Dirac measure at $x$ with total mass equal to the volume of the dual polytope $\{x\}^f$ with respect to the Haar measure $\lambda_M$ on $M_\R$ normalized such that the lattice $M$ has covolume $1$. Here, the dual polytope of $f$ is defined by  
$\{x\}^f=\{u\in M_\R \mid \textit{$f(\omega)-f(x)\geq\langle
\omega-x,u\rangle$ for all $\omega \in N_\R$}\}$. 
We refer to \cite[\S 2.7]{bps-asterisque} for details (replacing convex functions by concave functions).

\subsection{Non-archimedean geometry} \label{subsection nonargeom}

A  \emph{non-archimedean field} is
 a field $K$  complete with respect to
a given ultrametric absolute value $|\phantom{a}|\colon K\to \R_{\geq 0}$. The  \emph{valuation} is $v \coloneqq -\log|\phantom{a}|$ and $\Gamma \coloneqq v(K^\times)$ is the \emph{value group}. The \emph{valuation ring} is denoted by $K^\circ \coloneqq \{\alpha \in K \mid v(\alpha)\geq 0\}$ with maximal ideal $K^{\circ\circ}\coloneqq \{\alpha \in K \mid v(\alpha)> 0\}$ and \emph{residue field} $\ktilde \coloneqq \kcirc/K^{\circ \circ}$. 

We  consider  \emph{good} non-archimedean analytic spaces as introduced by Berkovich in \cite{berkovich-book}. They are characterized by the fact that every point has an affinoid neighbourhood. We are occupied with \emph{strictly analytic spaces} where we can use a closed analytic subspace of a unit ball for this affinoid analytic neighbourhood. We assume that the reader is familiar with the  notions from \cite{berkovich-book}.
We apply this to 
the \emph{analytification}  $\Xan$ of a variety $X$ over $K$. 

For a point $x$ in a strictly analytic Berkovich space $X$, we have the completed residue field $\mathscr H(x)$. We call $x$ an \emph{Abhyankar point} if the transcendence degree of the graded residue field of $\mathscr H(x)$ over $K$ is equal to the local dimension of $x$ at $X$ (in general, we have ``$\leq$''
 which is Abhyankar's inequality). 
 \walter{In classical terms, this graded transcendence degree   is  $r+d$, where $r$ is the transcendence degree 
 	of the usual reductions $(\mathscr H(x))^{\sptilde}$ 
 	over $\ktilde$  and $d$ is the dimension of the $\Q$-vector space 
 	$(|\mathscr H(x)^\times|\otimes \Q) / (|K^\times|\otimes \Q)$ build from the multiplicative value groups.}
The important point is that for a point $y$ in a closed analytic subspace $Y$ of $X$, the completed residue field $\mathscr H(y)$ is the same for $Y$ as for $X$ and hence if the local dimension of $Y$ at $y$ is strictly smaller than the dimension of $X$ at $y$, it follows that $y$ is not an Abhyankar point of $X$. We refer to \cite[\S 1.4]{ducros18:families} for details.

\subsection{Formal models and semipositive metrics}  \label{subsection fomospm}

We consider a non-trivially valued algebraically closed non-archimedean field $K$.
Let $\Xfrak$ be an \emph{admissible formal scheme over $\kcirc$} which means that $\Xfrak$ is a flat formal scheme over $\kcirc$ locally of topologically finite type such that $\Xfrak$ has a locally finite atlas by formal affine schemes over $\kcirc$. The \emph{generic fiber} of $\Xfrak$ is a paracompact strictly analytic Berkovich space over $K$ which we will denote by $\Xfrak_\eta$. 
\walter{This generic fiber is not necessarily a good analytic space, but we will use it later only in the good case as in the algebraic situation below.}
The \emph{special fiber} $\Xfrak_s$ is a scheme locally of finite type over $\kcirc$ and we have a \emph{reduction map} $\red\colon \Xfrak_\eta \to \Xfrak_s$. For details, we refer to \cite[\S 1.6]{berkovich-ihes} and \cite{bosch14:lectures_formal_rigid_geometry}.

\walter{Let $X$ be a paracompact strictly analytic space over $K$.  
A \emph{formal $\kcirc$-model} of $X$ is an admissible formal scheme $\Xfrak$ over $\kcirc$ with an identification $\Xfrak_\eta=X$. Let $L$ be a line bundle on $X$. Then a \emph{formal model $\Lfrak$ of $L$} is a line bundle $\Lfrak$ on a formal $\kcirc$-model $\Xfrak$ of $X$ such that $\Lfrak|_{\Xfrak_\eta}= L$ along the identification $\Xfrak_\eta=X$. We call $\Lfrak$ \emph{nef} if the line bundle $\Lfrak|_{\Xfrak_s}$ is nef on the special fiber $\Xfrak_s$. The latter means by definition that for every closed curve $Y$ of $\Xfrak_s$, which is proper over the residue field $\ktilde$, we have $\deg_\Lfrak(Y)\geq 0$.}

\walter{A formal model $\Lfrak$ of $L$ induces a continuous metric $\metr_\Lfrak$ on $L$, see \cite[Definition 2.5]{gubler-martin}. A metric $\metr$ on $L$ is called a \emph{model metric} if there is a non-zero $k \in \N$ such that $\metr^{\otimes k}$ is induced by a formal model $\Lfrak$ of $L^{\otimes k}$. We call a model metric \emph{nef} if $\Lfrak$ can be chosen as a nef line bundle. 
A continuous metric on $L$ is called \emph{semipositive} if it is a uniform limit of nef model metrics on $L$.}

\walter{In this paper, we work mainly in the algebraic setting where $X=Y^\an$ for a proper algebraic variety $Y$ over $K$. 
 Then one can replace formal $\kcirc$-models by proper algebraic models of $Y$ over $\kcirc$ (see \cite[\S 2]{gubler-martin}), but working with formal models allows additional flexibility and is more convenient.}

%Let $X$ be a proper algebraic variety over $K$ of dimension $n$. A \emph{formal $\kcirc$-model} of $X$ is an admissible formal scheme $\Xfrak$ over $\kcirc$ with an identification $\Xfrak_\eta=\Xan$. Let $L$ be a line bundle on $X$. Then a \emph{formal model $\Lfrak$ of $L$} is a line bundle $\Lfrak$ on a formal $\kcirc$-model $\Xfrak$ of $X$ such that $\Lfrak|_{\Xfrak_\eta}= \Lan$ along the identification $\Xfrak_\eta=\Xan$. We call $\Lfrak$ \emph{nef} if the line bundle $\Lfrak|_{\Xfrak_s}$ is nef on the special fiber $\Xfrak_s$. More generally, these notions can be used for any paracompact strictly analytic space replacing $\Xan$. In the algebraic setting, one can replace formal $\kcirc$-models by proper algebraic models over $\kcirc$ (see \cite[\S 2]{gubler-martin}), but working with formal models allows additional flexibility and is more convenient.

\subsection{Non-archimedean Monge--Amp\`ere measures} \label{subsection nonarMA}
\walter{Let $X$ be a proper algebraic variety over $K$ and let $L$ be a line bundle over $X$.}
A construction originated by Chambert--Loir \cite{chambert-loir-2006} associates to a model metric $\metr$ of $L$ a discrete measure $c_1(L,\metr)^{\wedge n}$ on $\Xan$. This can be used to define the \emph{Monge--Amp\`ere measure} $c_1(L,\metr)^{\wedge n}$ for any continuous semipositive metric $\metr$ of $\Lan$ by using that for a uniform limit of nef model metrics $\metr_k$ on $\Lan$ the corresponding sequence of measures $c_1(L,\metr_k)^{\wedge n}$ converges weakly in the sense of positive Radon measures, see \cite[\S 2]{gubler-2007b}. 

We briefly recall the construction of the Monge--Amp\`ere measure for $\metr_\Lfrak$. Since $K$ is algebraically closed, the formal models of $\Xan$ with reduced special fiber are cofinal \walter{in the category of all formal models of $\Xan$ with respect to morphisms extending the identity of $\Xan$ \cite[Proof of Proposition 3.5]{gubler-martin}.} \walter{This form of the reduced fiber theorem and the projection formula allow us to}  assume that $\Lfrak$ is a line bundle on a formal $\kcirc$-model with $\Xfrak_s$ reduced. Then for every irreducible component $Y$ of $\Xfrak_s$, there is a unique point $\xi_Y \in \Xan$ such that $\red{(\xi_Y)}$ is the generic point of $\Xfrak_s$. Such points are called \emph{Shilov points} for $\Xfrak_s$. We set
\begin{equation*} \label{definition Chambert-Loir measure}
c_1(L,\metr_\Lfrak)^{\wedge n} \coloneqq \sum_Y \deg_\Lfrak(Y) \cdot \delta_{\xi_Y}
\end{equation*}
where $Y$ ranges over all irreducible components of $\Xfrak_s$ and where $\delta_{\xi_Y}$ is the Dirac measure in the Shilov point $\xi_Y$.

\section{Piecewise linear approximation} \label{section plappr}

In this section, we prove that convex functions can be approximated by suitable generic piecewise linear functions in a setup later used for tropicalizations of abelian varieties.

The setup is as follows: Let $\Gamma$ be a non-trivial divisible subgroup of $\R$. In the applications, it will be the value group of a non-trivially valued algebraically closed non-archimedean field $K$. We consider a free abelian group $N$ of rank $n$ and a lattice $\Lambda$ in the base change $N_\R$ of $N$ to $\R$. 
%\walter{We assume that 
%\begin{equation} \label{lattice coordinates}
%\langle m, \lambda \rangle \in \Gamma
%\end{equation}
%holds for all $m \in M$ and all $\lambda \in \Lambda$, i.e.~that the lattice $\Lambda$ has coordinates in $\Gamma$ with respect to a $\Z$-basis in $M$.}
Later, these data will come naturally from the canonical tropicalization of an abelian variety over $K$. The dual group of $N$ is denoted by $M \coloneqq \Hom(N,\Z)$. 
\walter{For $m \in M_\R$ and $\omega \in N_\R$, we set $\langle m, \omega \rangle \coloneqq m(\omega)\in \R$.}

We also fix a  set $\Sigma$ of polytopes in $N_\R$. We assume that for $\Delta \in \Sigma$, all faces of $\Delta$ are also included in $\Sigma$. 
\walter{In the following, we will consider a locally finite polytopal decomposition $\Ccal$ of $N_\R$ (or more generally of a piecewise linear space) which means a polyhedral complex $\Ccal$ consisting of polytopes such that the support of $\Ccal$ is the whole ambient space $N_\R$. For a polytope $\sigma$ in $N_\R$, we will use the  linear subspace $\L_\sigma$ of $N_\R$ generated by $\{\omega-\nu \mid \omega,\nu \in \sigma\}$  and the  affine subspace $\A_\sigma$ of $N_\R$ generated by $\sigma$.}

\begin{definition} \label{transversal decomposition}
	A locally finite polytopal decomposition $\Ccal$ of $N_\R$ is called \emph{$\Sigma$-transversal} if for all $\sigma \in \Sigma$ and all $\Delta \in \Ccal$ with $\sigma \cap \Delta \neq \emptyset$, we have 
\begin{equation} \label{transversality for polytopes}
\dim(\sigma \cap \Delta)=\dim(\Delta)+\dim(\sigma)-n.
\end{equation}
\end{definition}

\begin{rem} \label{expected dimension}
Let $D(\sigma,\Delta) \coloneqq \dim(\Delta)+\dim(\sigma)-n$. Recall from linear algebra the dimension formula 
\begin{equation} \label{dimension formula}
\dim(\L_\sigma \cap \L_\Delta)= \dim(\Delta)+\dim(\sigma)-\dim(\L_\sigma+\L_\Delta)
\end{equation}
for the underlying linear spaces. Transversal intersection of $\L_\sigma$ and $\L_\Delta$ usually means that $\L_\sigma+\L_\Delta = N_\R$ which is equivalent to $\dim(\L_\sigma \cap \L_\Delta)=D(\sigma,\Delta)$. 
\end{rem}

\begin{lemma} \label{transversality lemma}
A locally finite polytopal decomposition $\Ccal$ of $N_\R$ is $\Sigma$-transversal if it satisfies the following two conditions for all $\sigma \in \Sigma$ and $\Delta \in \Ccal$ with underlying affine spaces $\A_\sigma, \A_\Delta$:
\begin{enumerate}
	\item \label{first condition for transversality}
	If $D(\sigma,\Delta)\geq 0$, then $\L_\sigma + \L_\Delta = N_\R$.
	\item \label{second condition for transversality}
	If $D(\sigma,\Delta)<0$, then $\A_\sigma \cap \A_\Delta =\emptyset$.
\end{enumerate}
\end{lemma}

\begin{proof}
The argument follows \cite[Proposition 8.2]{gubler-2007a}.
	Assume that $\sigma \cap \Delta \neq \emptyset$. Then by \ref{first condition for transversality} and \ref{second condition for transversality}, we have $\L_\sigma+\L_\Delta=N_\R$. Using $\sigma \cap \Delta \neq \emptyset$ and Remark \ref{expected dimension}, we get 
	\begin{equation} \label{conclusion from two assumptions}
	\dim(\A_\sigma \cap \A_\Delta)= 
	 \dim(\L_\sigma \cap \L_\Delta)=D(\sigma,\Delta).  
	\end{equation}
	
	If $\relint(\sigma) \cap \relint(\Delta)\neq\emptyset$, then we have $\dim(\Delta \cap \sigma)=\dim(\A_\sigma \cap \A_\Delta)$ and \eqref{transversality for polytopes} follows from \eqref{conclusion from two assumptions}.
	It remains to see that $\relint(\sigma) \cap \relint(\Delta)=\emptyset$ cannot happen. We argue by contradiction. We may assume that $\Delta$ and $\sigma$ are minimal with $\relint(\sigma) \cap \relint(\Delta)=\emptyset$. Using that the roles of $\sigma$ and $\Delta$ are symmetric, we may assume that there is a proper face $\sigma'$ of $\sigma$ of codimension $1$ with $\sigma' \cap \Delta \neq \emptyset$. Note that $\A_{\sigma'}$ divides $\A_\sigma$ into two half-spaces, and precisely one contains $\sigma$.
	By minimality, we have $\relint(\sigma') \cap \relint(\Delta)\neq\emptyset$. 
	Using also $\relint(\sigma) \cap \relint(\Delta)=\emptyset$, we deduce that $\A_\sigma \cap \relint(\Delta) \subset \A_{\sigma'}$ and hence 
	$\A_\sigma \cap \A_\Delta =  \A_{\sigma'} \cap \A_\Delta$. Since $\dim(\sigma')<\dim(\sigma)$, we have $D(\sigma',\Delta)<D(\sigma,\Delta)$ which contradicts \eqref{conclusion from two assumptions} applied to $\sigma$ and $\sigma'$.
\end{proof}

\begin{definition} \label{periodic decompositions}
A locally finite polytopal decomposition $\Ccal$  of $N_\R$ is called \emph{$\Lambda$-periodic} if for all $\Delta \in \Ccal$ and for all $\lambda \in \Lambda\setminus\{0\}$ the polytope $\Delta + \lambda$ is a face of $\Ccal$ disjoint from $\Delta$. 	
\end{definition}

These conditions ensure that we can see the image $\overline \Delta$ of $\Delta$ in $N_\R/\Lambda$ as a polytope in $N_\R/\Lambda$ and that the set of all $\overline \Delta$ is a polytopal decomposition $\overline \Ccal$ of $N_\R/\Lambda$.

\begin{art} \label{cocycle}
\walter{We fix now the following data for the remaining part of this section.}
We fix a positive definite inner product $b$ on $N_\R$
such that
\begin{equation} \label{rationality of b}
\walter{b(\lambda,x) \in \Z}
\end{equation}
for all $x\in N$ and for all $\lambda \in \Lambda$. 
In the applications, such a bilinear form $b$ will be induced by an ample line bundle on the abelian variety $A$. 
%\walter{We first show that 
%\begin{equation} \label{lattice values}
%b(\lambda, y) \in \Gamma
%\end{equation}
%holds for all $y,\lambda \in \Lambda_\Q$. In fact, assuming \eqref{rationality of b}, the argument will show that the conditions \eqref{lattice coordinates} and \eqref{lattice values} are equivalent. Indeed,  non-degeneracy of $b$ and  \eqref{rationality of b} show that 
%$$\Lambda_\Q \longrightarrow M_\Q, \quad \lambda \longmapsto b(\lambda,\cdot)$$
%is an isomorphism of $\Q$-vector spaces of dimension $n$. For $y \in \Lambda_\Q$, we conclude that there is $m \in M_\Q$ with $\langle m, \cdot \rangle = b(y,\cdot)$. Inserting $\lambda \in \Lambda_\Q$, divisibility of $\Gamma$ and \eqref{lattice coordinates} yield \eqref{lattice values}.}

We also consider a \emph{$\Lambda$-cocycle} $(z_\lambda)_{\lambda \in \Lambda}$ on $N_\R$, i.e.~functions $z_\lambda\colon N_\R \to \R$ satisfying 
\begin{equation} \label{cocycle condition}
z_{\lambda+\nu}(\omega)=z_\lambda(\omega+\nu)+z_\nu(\omega)
\end{equation}
for all $\lambda,\nu \in \Lambda$ and all $\omega \in N_\R$. We assume that the properties
\begin{equation} \label{symmetric bilinear form}
z_\lambda(\omega)=z_\lambda(0)+b(\lambda,\omega)
\end{equation}
and 
\begin{equation} \label{quadratic rationality property}
z_\lambda(0) \in \Gamma
\end{equation}	
hold for 
 all $\lambda \in \Lambda$ and $\omega \in N_\R$. 
 Using \eqref{cocycle condition} and \eqref{symmetric bilinear form}, we see that 
$$z_{\lambda + \nu}(0)-z_\lambda(0)-z_\nu(0)=z_\lambda(\nu)-z_\lambda(0)=b(\lambda, \nu).$$
It follows that $\walter{z_\lambda(0)}$ is a quadratic function \walter{in $\lambda$} with associated symmetric bilinear form $b$.
\end{art}

\walter{It is in \eqref{quadratic rationality property} that the fixed divisible subgroup $\Gamma$ of $\R$ shows up first.}
\walter{\begin{lem} \label{rationality consequences}
The assumptions on the bilinear form $b$ and on the cocycle $(z_\lambda)_{\lambda \in \Lambda}$ in \ref{cocycle} yield  
$b(\lambda, y) \in \Gamma$ and $\langle m, \lambda \rangle \in \Gamma$  for all $y,\lambda \in \Lambda_\Q$ and $m \in M$. 
\end{lem}	
\begin{proof}
Since $z_\lambda(0)$ is a quadratic function in $\lambda \in \Lambda$ with associated symmetric bilinear form $b$, we deduce that 
$$z_\lambda(0)+z_{-\lambda}(0)= b(\lambda,\lambda)$$
and hence we deduce from \eqref{quadratic rationality property} that $b(\lambda,\lambda) \in \Gamma$. Using that $\Gamma$ is divisible, this holds even for all $\lambda \in \Lambda_\Q$. For any $y \in \Lambda_\Q$ and using again that $\Gamma$ is divisible, we have 
$$b(\lambda,y)=\frac{1}{2}\left(b(\lambda+y,\lambda+y) -b(\lambda,\lambda)-b(y,y)\right) \in \Gamma.$$
The non-degeneracy of $b$ and  \eqref{rationality of b} show that 
$$\Lambda_\Q \longrightarrow M_\Q, \quad \lambda \longmapsto b(\lambda,\cdot)$$
is an isomorphism of $\Q$-vector spaces of dimension $n$. For  $m \in M_\Q$, we conclude that there is  $\lambda \in \Lambda_\Q$ with $\langle m, \cdot \rangle = b(\lambda,\cdot)$. For $y \in \Lambda_\Q$, we get $\langle m, y \rangle = b(\lambda, y) \in \Gamma$.
\end{proof}}
 
\begin{definition} \label{cocycle rule for function}
\walter{Using a fixed cocycle $(z_\lambda)_{\lambda \in \Lambda}$ as above,} we say that $f\colon N_\R \to \R$ satisfies the \emph{cocycle rule} if 
$$f(\omega+\lambda)= f(\omega)+ z_\lambda(\omega)$$
for all $\omega \in N_\R$ and $\lambda \in \Lambda$. 
\end{definition}

These functions may be seen as tropical theta-functions, see \cite{mikhalkin-zharkov2008} and \cite{foster-rabinoff-shokrieh-soto}. 

\walter{In the following, we will call a piecewise linear function $f \colon N_\R \to \R$ \emph{strictly convex with respect to a locally finite polytopal decomposition $\Ccal$ of $N_\R$} if the function $f$ is convex and if for all $\sigma \in \Ccal$ we have $f|_\sigma = u_\sigma + c_\sigma$ with $u_\sigma \in M_\R$ and $c_\sigma \in \R$ such that the slopes $u_\sigma$ are different for different maximal faces $\sigma \in \Ccal$.
	This is similar to the notion used in toric geometry \cite{fulton-toric-varieties} and should not be confused with the notation of strictly convex functions in analysis which has a different meaning.} 

We can now state the main result of this section.

\begin{prop} \label{transversal pl approximation}
Let $\Sigma$ be a finite set of polytopes in $N_\R$. We assume that $\Sigma$ includes with a polytope also all its faces. Under the above assumptions, we consider a convex function $f\colon N_\R \to \R$ satisfying the cocycle rule from Definition \ref{cocycle rule for function}. Then $f$ is the uniform limit of  functions $f_k$ satisfying the same cocycle rule  such that every  $f_k$ is a piecewise $(\Q,\Gamma)$-linear strictly convex function with respect to a locally finite $(\Z,\Gamma)$-polytopal decomposition $\Ccal_k$ of $N_\R$ which is $\Lambda$-periodic and $\Sigma$-transversal.
\end{prop}

\walter{For the proof, we need a couple of lemmas first.}

\begin{lem} \label{convex piecewise linear functions}
	Let $f$ be a convex piecewise linear function satisfying the cocycle rule with respect to the given cocycle $z=(z_\lambda)_{\lambda \in \Lambda}$ as above. Then the following properties hold:
	\begin{enumerate}
		\item \label{maximal domains form decomposition}
		The maximal domains of linearity are the $n$-dimensional faces of a locally finite polytopal decomposition $\Dcal(f)$ of $N_\R$.
		\item \label{rationality of maximal domains}
		If $f$ is piecewise $(\Q,\Gamma)$-linear, then $\Dcal(f)$ is a $(\Z,\Gamma)$-polytopal decomposition.
		\item \label{translation of maximal domains}
		For $\Delta \in \Dcal(f)$ and $\lambda \in \Lambda$, we have  $\Delta+\lambda \in \Dcal(f)$. 
		\item \label{almost periodicity}
		For a maximal domain of linearity $\Delta$ of $f$ and a non-zero $\lambda \in \Lambda$, the open sets $\relint(\Delta)+\lambda$ and  $\relint(\Delta)$ are disjoint.
		%\item \label{slope condition for fail of periodicity}
		%If  $\Delta \cap (\Delta+\lambda)\neq \emptyset$ for some $\Delta \in \Dcal(f)$ and a non-zero $\lambda \in \Lambda$, then the slope of $f|_\Delta$ agrees with $b(\lambda,\cdot)$.
	\end{enumerate}
	\end{lem}

This shows that $\Dcal(f)$ is almost a $\Lambda$-periodic polytopal decomposition of $N_\R$, but we can only ensure that $(\Delta+\lambda)\cap \Delta$ is contained in the relative boundary of $\Delta$.

\begin{proof}
A \emph{convex} piecewise linear function can be written as  the maximum (supremum) of  affine functions. This shows that the maximal domains of linearity are the $n$-dimensional faces of a locally finite \emph{polyhedral} decomposition $\Dcal(f)$ of $N_\R$. If $f$ is piecewise $(\Q,\Gamma)$-linear, then this also shows that every $\Delta \in \Dcal(f)$ is $(\Z,\Gamma)$-linear. 

Let $\Delta$ be a maximal domain of linearity for $f$ and let $\lambda \in \Lambda$. Then assumption \eqref{symmetric bilinear form} shows that $\Delta+ \lambda$ is a (maximal) domain of linearity for $f$ proving 
 \ref{translation of maximal domains}. 

It remains to show that $\Delta$  is a polytope and that property \ref{almost periodicity} holds for $\Delta$.  Using the cocycle rule and that the coycles $z_\lambda$ grow quadratically in $\lambda \in \Lambda$, the polyhedron $\Delta$ is  indeed bounded and hence is a polytope. Since $\Delta+\lambda$ is also a domain of linearity for $f$, it is clear for $\lambda \in \Lambda \setminus \{0\}$ that $\relint(\Delta)+ \lambda$ is disjoint from $\relint(\Delta)$ as otherwise they would agree, hence $\Delta+\lambda=\Delta$ by maximality, and an inductive argument adding successively $\lambda$ would show that $\Delta$ is unbounded.
\end{proof}	

In the following, we fix a polytopal decomposition $\Ccal$ of $N_\R$ assuming that the lattice $\Lambda$ acts on $\Ccal$ by translation. 
 Recall that the set of maximal faces of $\Ccal$ is denoted by $\Ccal_n \coloneqq \{\sigma \in \Ccal \mid \dim(\sigma)=n\}$. We choose a fixed set of representatives $\Ncal$ for the $n$-dimensional faces of $\Ccal$ with respect to this $\Lambda$-action.

%\walter{Before we proceed with the proof, we need a couple of lemmas where the cocycle $(z_\lambda)_{\lambda \in \Lambda}$ on $N_\R$ is fixed and where $\Ncal$ is a fixed set of representatives of the maximal faces of a locally finite $\Lambda$-periodic polytopal decomposition $\Ccal$ of $N_\R$ with respect to the $\Lambda$-action on $N_\R$. Recall that the set of maximal faces of $\Ccal$ is denoted by $\Ccal_n \coloneqq \{\sigma \in \Ccal \mid \dim(\sigma)=n\}$.}
\begin{lem} \label{polytopes and representatives}
		A polytope $\Delta'$ is in $\Ccal_n$ if and only if it has the form 
		\begin{equation} \label{system of representatives for Ccal}
		\Delta' = \Delta + \lambda
		\end{equation}
		with $\Delta \in \Ncal$ and $\lambda \in \Lambda$. Moreover,  $\Delta$ and  $\lambda$ are uniquely determined by $\Delta'$.
\end{lem}

\begin{proof}		 
		This follows  from the definition of a system of representatives for a $\Lambda$-action.
\end{proof}
Recall that $b$ is a positive definite inner product on $N_\R$ satisfying the assumptions in \ref{cocycle}.
\begin{lem} \label{periodicity of piecewise linear functions}
		Let $f \colon N_\R \to \R$ be a piecewise linear function with respect to $\Ccal$, i.e.~for $\Delta \in \Ccal$, there is a slope $m_\Delta \in M_\R$ and a constant term $c_\Delta \in \R$ such that 
		\begin{equation} \label{slope and constant term}
		f = m_\Delta + c_\Delta
		\end{equation}
		on $\Delta$. Then $f$ satisfies the cocycle rule if and only if
		the condition
		\begin{equation} \label{cocycle rule and slopes}
		m_{\Delta+\lambda}= m_\Delta + b(\lambda,\cdot)
		\end{equation}
		for the slopes and the condition
		\begin{equation} \label{cocycle rule and constant terms}
		c_{\Delta+\lambda}=c_\Delta-\langle m_\Delta, \lambda \rangle +z_\lambda(0)-b(\lambda,\lambda)
		\end{equation}
		for the constant terms hold for all $\Delta \in \Ncal$ and for all $\lambda \in \Lambda$. In this case, $f$ is a piecewise $(\Q,\Gamma)$-linear function if and only if $m_\Delta \in M_\Q$ and $c_\Delta \in \Gamma$ for all $\Delta \in \Ncal$.
\end{lem}
%For $\Delta \in \Ccal$, let $m_\Delta \in M_\R$ be the slope of $f|_\Delta$ and $c_\Delta$ the constant term such that 
%\begin{equation} \label{slope and constant term}
%f = m_\Delta + c_\Delta
%\end{equation}
%on $\Delta$. Recall that $\Ccal_n \coloneqq \{\sigma \in \Ccal \mid \dim(\sigma)=n\}$. Let $\Ncal$ be a system of representatives for $\Ccal_n$  with respect to the $\Lambda$-action, i.e.~a polytope $\Delta'$ is in $\Ccal_n$ if and only if it has the form 
%\begin{equation} \label{system of representatives for Ccal}
%\Delta' = \Delta + \lambda
%\end{equation}
%with $\Delta \in \Ncal$ and $\lambda \in \Lambda$. Moreover,  $\Delta$ and  $\lambda$ are uniquely determined by $\Delta'$. 
%%For the finitely many $\Delta \in \Ncal$, we  will approximate $m_\Delta$ by $m_\Delta'\in M_\Q$ and $c_\Delta$ by $c_\Delta' \in \Gamma$ using the density of $\Gamma$ in $\R$. 
\begin{proof}
	Using that $\Ccal$ is $\Lambda$-periodic and Lemma \ref{polytopes and representatives}, it is clear that  $f$ satisfies the cocycle rule if and only if for any $\Delta \in \Ncal$, we have
	\begin{equation} \label{cocycle rule and affine functions}
	\langle m_{\Delta+\lambda}, \omega \rangle + \langle m_{\Delta+\lambda}, \lambda \rangle +c_{\Delta+\lambda}
	= \langle m_\Delta , \omega \rangle +c_\Delta + z_\lambda(\omega)
	\end{equation}
	for any $\lambda \in \Lambda$ and $\omega \in N_\R$. Equivalently, using $z_\lambda(\omega)=z_\lambda(0)+b(\lambda,\omega)$ from \eqref{symmetric bilinear form}, we have
	the equations \eqref{cocycle rule and slopes} and \eqref{cocycle rule and constant terms}. 
	\walter{Finally, the last claim follows easily from  \eqref{cocycle rule and slopes} and \eqref{cocycle rule and constant terms} by using  \eqref{rationality of b}, \eqref{quadratic rationality property} and Lemma \ref{rationality consequences}.} 
	%\begin{equation} \label{cocycle rule and slopes}
	%m_{\Delta+\lambda}= m_\Delta + b(\lambda,\cdot)
	%\end{equation}
	%for the slopes and
	%\begin{equation} \label{cocycle rule and constant terms}
	%c_{\Delta+\lambda}=c_\Delta+z_\lambda(0)-\langle m_{\Delta+\lambda}, \lambda \rangle =c_\Delta-\langle m_\Delta, \lambda \rangle +z_\lambda(0)-b(\lambda,\lambda)
	%\end{equation}
	%and for the constant terms. We use this as  a guideline to define
	%Now assume that $f$ satisfies the cocycle rule. 
	%If  $m_\Delta \in M_\Q$ for $\Delta \in \Ncal$, then \eqref{cocycle rule and slopes} and \eqref{rationality of b} show that $m_{\Delta+\lambda} \in M_\Q$. If $c_\Delta \in \Gamma$, then it follows from 
	%\eqref{cocycle rule and constant terms}, \eqref{rationality of b} and \eqref{quadratic rationality property} that $c_{\Delta + \lambda}$ is an element of the divisible group $\Gamma$. This proves readily the last claim.
\end{proof} 

Let $(m,c) \coloneqq (m_\Delta,c_\Delta)_{\Delta \in \Ncal} \in (M_\R \times \R)^\Ncal$. For $\Delta' \in \Ccal_n$, Lemma \ref{polytopes and representatives} shows that there are uniquely determined $\lambda \in \Lambda$ and $\Delta \in \Ncal$ such that $\Delta' = \Delta + \lambda$. Inspired by Lemma \ref{periodicity of piecewise linear functions}, we define
\begin{equation} \label{periodic extensions of slopes and constant terms}
m_{\Delta'} \coloneqq m_{\Delta+\lambda}\coloneqq m_\Delta + b(\lambda,\cdot)\quad \text{and} \quad c_{\Delta'} \coloneqq c_{\Delta+\lambda} \coloneqq c_\Delta-\langle m_\Delta, \lambda \rangle +z_\lambda(0)-b(\lambda,\lambda)
\end{equation}
and then 
\begin{equation} \label{resulting PL convex function}
f_{(m,c)} \coloneqq \sup \{ m_{\Delta+\lambda} + c_{\Delta + \lambda} \mid \Delta \in \Ncal, \, \lambda \in \Lambda\}.
\end{equation}

\begin{lem} \label{approximation function}
	For every  $(m,c) \coloneqq (m_\Delta,c_\Delta)_{\Delta \in \Ncal}\in (M_\R \times \R)^\Ncal$, the above procedure defines a convex piecewise linear $f_{(m,c)}$ satisfying the cocycle rule. If $(m,c) \in (M_\Q \times \Gamma)^\Ncal$, then $f_{(m,c)}$ is a piecewise $(\Q,\Gamma)$-linear function. 
\end{lem}

\begin{proof}
	For $\omega$ in a bounded domain of $N_\R$, the number $\langle m_{\Delta+\lambda}, \omega \rangle$ grows at most linearly in $\lambda \in \Lambda$, see \eqref{cocycle rule and slopes}. Since $z_\lambda(0)$ is a quadratic function in $\lambda$ with associated bilinear form $b$, we see that $2z_\lambda(0) - b(\lambda,\lambda)$ is a linear function in $\lambda$ and hence \eqref{cocycle rule and constant terms} shows that the term $c_{\Delta+\lambda}$ decreases quadratically in $\lambda$. We conclude that on a bounded domain in $N_\R$ only finitely many $\lambda \in \Lambda$ contribute to the supremum in \eqref{resulting PL convex function} and hence $f_{(c,m)}$ is a piecewise linear function. By construction, the slopes and the constant terms of $f_{(c,m)}$ transform as in \eqref{cocycle rule and slopes} and in \eqref{cocycle rule and constant terms}, respectively. By Lemma \ref{periodicity of piecewise linear functions}, the function $f_{(c,m)}$ satisfies the cocycle rule. The last claim also follows from Lemma \ref{periodicity of piecewise linear functions}.
\end{proof}

\begin{rem} \label{only finitely many contributions}
	For a bounded domain $\Omega$ of $N_\R$ and any constant $R \in \R$, the above argument shows that 
	$$\langle m_{\Delta'}, \omega \rangle +c_{\Delta'}\leq f_{(m,c)}(\omega)-R$$
	for all $\omega \in \Omega$ and for all but finitely many $\Delta' \in \Ccal_n$. Indeed, we have $\Delta'=\Delta + \lambda$ with $\Delta \in \Ncal$ and $\lambda \in \Lambda$ by Lemma \ref{polytopes and representatives}. The argument in Lemma \ref{approximation function} shows that the slopes increase at most linearly in $\lambda$ while the constant terms decrease quadratically in $\lambda$. This proves the claim.
\end{rem}

Let $\PL(N_\R,\Lambda,z)$ be the space of piecewise linear functions on $N_\R$ satisfying the cocycle rule with respect to the given cocycle $z=(z_\lambda)_{\lambda \in \Lambda}$. For $f,g \in \PL(N_\R,\Lambda,z)$, the cocycle rule yields that $f-g$ is $\Lambda$-periodic and hence their distance
$$d(f,g) \coloneqq \sup_{\omega \in N_\R} |f(\omega)-g(\omega)|$$
is a well-defined real number defining a metric on $\PL(N_\R,\Lambda,z)$. On the finite dimensional $\R$-vector space $(M_\R \times \R)^\Ncal$, we will use any norm. 

\begin{lem} \label{continuous functional}
	With respect to the above metrics und using Lemma \ref{approximation function}, we get a uniformly continuous map 
	$$(M_\R \times \R)^\Ncal \longrightarrow \PL(N_\R,\Lambda,z), \quad (m,c)\longmapsto f_{(m,c)}.$$
\end{lem}

\begin{proof} 
	As all norms on a finite dimensional $\R$-vector space are equivalent, we may assume that the distance on $(M_\R \times \R)^\Ncal$ is the max-norm induced by the same norm $\metr$ on each factor $M_\R$ and by the standard norm on $\R$. 
	Let $(m,c)$ and $(m',c)$ be elements of $(M_\R \times \R)^\Ncal$ with distance $\delta$. By \eqref{periodic extensions of slopes and constant terms}, we deduce that 
	\begin{equation} \label{distance and translates}
	\|(m_{\Delta+\lambda}',c_{\Delta+\lambda}')-(m_{\Delta+\lambda},c_{\Delta+\lambda})\|=
	\|(m_{\Delta}',c_{\Delta}')-(m_{\Delta},c_{\Delta})\|
	\leq \delta
	\end{equation} 
	for all $\Delta \in \Ncal$ and all $\lambda \in \Lambda$. Using \eqref{resulting PL convex function}, this easily proves uniform continuity.
\end{proof}

\begin{art} \label{center and ball}
Let $f$ be a convex piecewise linear function satisfying the cocycle rule with respect to the given cocycle $z=(z_\lambda)_{\lambda \in \Lambda}$. Using Lemma \ref{convex piecewise linear functions}, we see that $f$ is strictly convex with respect to the polytopal decomposition $\Ccal \coloneqq \Dcal(f)$. Let $\Ncal$ be again a system of representatives for $\Ccal_n$ with respect to the $\Lambda$-action. Since $f$ is a convex piecewise linear function, it is the maximum of the affine pieces obtained from $\Ccal_n$ which shows that  there is $(m,c)\in (M_\R \times \R)^\Ncal$ such that $f=f_{(m,c)}$ as defined in \eqref{resulting PL convex function}.
%in the notation of Lemma \ref{convex piecewise linear functions}.
We use a fixed norm $\metr$ on $N_\R$. For $\delta > 0$ and $\Delta \in \Ccal_n$, we define the \emph{$\delta$-center} of $\Delta$ as 
$$C(\Delta,\delta) \coloneqq \{\omega \in \Delta \mid \text{$\|\omega - \omega'\|\geq \delta$ for all $\omega' \in N_\R \setminus \Delta$}\}.$$
It is clear that $C(\Delta,\delta)$ is a polytope contained in $\Delta$. 
Moreover,  we define $B(\Delta,\delta)$ as the set of points in $N_\R$ with distance $< \delta$ to $\Delta$. 
\end{art}

%For the next lemma, we consider a $\Lambda$-periodic polytopal decomposition $\Ccal$ of $N_\R$ as defined in \ref{periodic decompositions}. We choose again a system of representatives $\Ncal$ of $\Ccal_n$ with respect to the $\Lambda$-action. 

%\begin{lem} \label{periodic decomposition for approximations}
%	Let  $f\colon N_\R \to \R$ be a strictly convex piecewise linear function with respect to $\Ccal$ and we assume that $f$   satisfies the cocycle rule with respect to the given cocycle $z=(z_\lambda)_{\lambda \in \Lambda}$. 
%	Then there is $(m,c) \in (N_\R \times \R)^\Ncal$ such that $f=f_{(m,c)}$. Moreover, there is a neighbourhood $\Omega$ of $(m,c)$ in $(N_\R \times \R)^\Ncal$ such that for all $(m',c') \in \Omega$, the polytopal decomposition $\Ccal' \coloneqq \Dcal(f_{(m',c')})$ is $\Lambda$-periodic.
%\end{lem}

The next results describes the change of $f=f_{(m,c)}$ and the underlying polytopal complex $\Ccal=\Dcal(f)$ if we replace $(m,c)$ by sufficiently good approximations  in $(M_\R \times \R)^\Ncal$.

\begin{lem} \label{periodic decomposition for approximations}
	Let $f=f_{(m,c)}$ be a function and $\Ccal=\Dcal(f)$ a polytopal decomposition as in \ref{center and ball}. Then there is a $\delta>0$ and a neighbourhood $U$ of $(m,c)$ in $(M_\R \times \R)^\Ncal$ such that the following properties hold for any $(m',c') \in U$ and $\Ccal'\coloneqq\Dcal(f_{(m',c')})$:
	\begin{enumerate}
		\item \label{distance and maximality}
		If $f_{(m',c')}(\omega)= \langle m_\Delta',\omega \rangle + c_\Delta'$ for some $\omega \in N_\R$ and $\Delta \in \Ccal_n$, then $\omega \in B(\Delta,\delta)$.
        \item \label{non-empty centers}
        The $\delta$-centers $C(\Delta,\delta)$ are  $n$-dimensional polytopes for any $\Delta \in \Ccal_n$.
        \item \label{domains of linearity and center}
        For any $\Delta \in \Ccal_n$, there is a unique $n$-dimensional face of $\Ccal'$ containing $C(\Delta,\delta)$.
        \item \label{presentation of f}
        Using the notation from \ref{domains of linearity and center}, we have $f_{(m',c')}=m_\Delta'+c_\Delta'$ on $\Delta'$ and the map $\Delta \mapsto \Delta'$ is a  canonical bijection from $\Ccal_n$ onto $\Ccal_n'$.
		\item \label{periodicity}
		If $\Ccal$ is a $\Lambda$-periodic polytopal decomposition of $N_\R$ (see Definition \ref{periodic decompositions}), then $\Ccal'$ is also a  $\Lambda$-periodic polytopal decomposition of $N_\R$.
	\end{enumerate}
\end{lem}

%\begin{proof}
%	The existence of $(m,c)\in (N_\R \times \R)^\Ncal$ with $f=f_{(m,c)}$ follows as in Lemma \ref{approximations and decompositions}.  For $\varepsilon >0$, we define $B(\Delta,\epsilon)$ as the set of points in $N_\R$ with distance $\leq \varepsilon$ to $\Delta$ using a fixed norm $\metr$ on $N_\R$. Since $\Ccal$ is $\Lambda$-periodic, we can choose $\varepsilon$ so small that for all $\Delta \in \Ccal_n$ and all non-zero $\lambda \in \Lambda$, we have 
%	$$\left( B(\Delta,\varepsilon) + \lambda \right) \cap B(\Delta,\varepsilon) = \emptyset.$$
%	Similarly as in the proof of Lemma \ref{approximations and decompositions}, there is $r>0$ and a   neighbourhood $\Omega$ of $(m,c)$ in $(N_\R \times \R)^\Ncal$ such that for any $(m',c')\in \Omega$ and any $\Delta \in \Ccal_n$, we have
%	\begin{equation*} \label{center}
%	\langle m_\Delta', \omega \rangle + c_\Delta' =f'(\omega)>\langle m_\nu', \omega \rangle+c_\nu'+r
%	\end{equation*}
%	for all  $\nu \in \Ccal_n \setminus \{\Delta\}$ and all $\omega \in \Delta$ with distance $\geq \varepsilon$ to $\nu$. 
%\end{proof}

\begin{proof}
    Let us pick $\nu\in \Ccal_n$. Strict convexity of $f$ and Remark \ref{only finitely many contributions} applied to $\Omega=\nu$ yield that there is $r>0$ such that 
	\begin{equation} \label{lower bound and strict convexity}
    \langle m_\nu, \omega \rangle + c_\nu =f(\omega)>\langle m_\Delta, \omega \rangle+c_\Delta+2r
    \end{equation}
	for all $\Delta \in \Ccal_n \setminus \{\nu\}$ and all $\omega \in \nu$ with distance $\geq \delta$ to $\Delta$. 
	%for all $\omega \in C(\sigma)$ and all $\Delta \in \Ccal_n \setminus \{\sigma\}$. 
	We may choose $r$ so small that the inequality holds for all $\nu \in \Ccal_n$ simultaneously. Indeed, this is clear for the representatives $\nu$ in the finite set $\Ncal$ and then by \eqref{cocycle rule and slopes} and \eqref{cocycle rule and constant terms} for all $\Lambda$-translates. Using \eqref{lower bound and strict convexity} and \eqref{distance and translates}, we conclude for all  $(m',c')$ in a sufficiently small neighbourhood $U$ of $(m,c)$ in $(M_\R \times \R)^\Ncal$ that for any $\Delta \in \Ccal_n$, we have
	\begin{equation} \label{center}
	\langle m_\nu', \omega \rangle + c_\nu' =f_{(m',c')}(\omega)>\langle m_\Delta', \omega \rangle+c_\Delta'+r
	\end{equation}
	for all  $\Delta \in \Ccal_n \setminus \{\nu\}$ and all $\omega \in \nu$ with distance $\geq \delta$ to $\Delta$.

    Now suppose that $\omega \in N_\R$ is as in \ref{distance and maximality}. This means that $\langle m_\Delta', \omega \rangle+c_\Delta'$ is maximal among all $\langle m_\nu', \omega \rangle+c_\nu'$ with $\nu$ ranging over $\Ccal_n$. There is $\nu \in \Ccal_n$ with $\omega \in \nu$. It follows from \eqref{center} that $\omega \in B(\Delta,\delta)$ proving \ref{distance and maximality}.

	For a fixed $\Delta \in \Ccal_n$, it is clear that the $\delta$-center $C(\Delta,\delta)$ is an $n$-dimensional polytope contained in $\Delta$ if we choose $\delta>0$ sufficiently small. Obviously, we can choose such a $\delta$ which works for the finitely many $\Delta \in \Ncal$. Since $\Ncal$ is a system of representatives for $\Ccal_n$ with respect to the $\Lambda$-translation, this number $\delta$ works for all $\Delta \in \Ccal_n$ proving \ref{non-empty centers}.
	
	We still consider $(m',c')\in U$. For $\Delta \in \Ccal_n$, let $\Delta'$ be the locus where  $f_{(m',c')}=m_\Delta'+c_\Delta'$. It follows from \ref{distance and maximality} that $C(\Delta,\delta)\subset \Delta'$. Since $C(\Delta,\delta)$ is an $n$-dimensional polytope, it is clear that $\Delta'$ is an $n$-dimensional face of $\Ccal'=\Dcal(f_{(m',c')})$ proving \ref{domains of linearity and center} and the first claim in \ref{presentation of f}. The definition of $f_{(m',c')}$ based on \eqref{resulting PL convex function} shows that $f_{(m',c')}$ is given on any maximal domain of linearity by $f_{(m',c')}=m_\Delta'+c_\Delta'$ for some $\Delta \in \Ccal_n$. This proves surjectivity of the map in \ref{presentation of f} and injectivity follows from \eqref{center}.
	
	Now we assume that $\Ccal=\Dcal(f)$ is a $\Lambda$-periodic polytopal decomposition of $N_\R$. To prove \ref{periodicity}, we may choose $\delta>0$ so small that
	\begin{equation} \label{disjointness of domain balls} 
	B(\Delta,\delta) \cap (B(\Delta,\delta)+\lambda) =\emptyset
	\end{equation} 
	for all non-zero $\lambda \in \Lambda$ and all $\Delta \in \Ccal_n$. This is possible for a single $\Delta \in \Ccal_n$ by using  $\Delta \cap (\Delta+\lambda)=\emptyset$ based on the definition of $\Lambda$-periodicity, then obviously also for the finite set of representatives $\Ncal$ for $\Ccal_n$ with respect to the $\Lambda$-translation and hence for all $\Delta \in \Ccal_n$ by the usual translation argument. Let $\Delta' \in \Ccal_n'$. By definition of $\Ccal'=\Dcal(f_{(m',c')})$, there is a $\Delta \in \Ccal_n$ such that $f_{(m',c')}= m_\Delta' + c_\Delta'$ on $\Delta'$. As seen in \ref{distance and maximality}, we have $\Delta' \subset B(\Delta,\delta)$ and hence \eqref{disjointness of domain balls} proves $\Delta' \cap (\Delta'+\lambda)=\emptyset$ showing \ref{periodicity}.
	%	We conclude that $f_{(m',c')}$ is given on the $\delta$-center $C(\Delta,\delta)$ by $m_\Delta'+c_\Delta'$ which
%	proves \eqref{approximation on center}. Then there is a 
%	unique $\Delta' \in \Ccal_n'$ containing the $\delta$-center $C(\Delta,\delta)$ and the map $\Delta \to \Delta'$ gives a bijection $\Ccal_n \to \Ccal_n'$. 
	%By construction, the Hausdorff distance of the domains of linearity $\sigma$ and $\sigma'$ is smaller than $\delta$. Using $\delta>0$ sufficiently small and using that $\Ccal$ is $\Lambda$-periodic, we conclude that $\Ccal'$ is $\Lambda$-periodic. Indeed, the cocycle rule for $f'$ shows that $\sigma' \in \Ccal_n'$ yields $\sigma'+\lambda \in \Ccal_n'$ for any non-zero $\lambda \in \Lambda$ and using that $\sigma$ has a positive Hausdorff distance to $\sigma+\lambda$, the same has to be true for $\sigma'$ and $\sigma'+\lambda$,  proving that $\sigma'$ and $\sigma'+\lambda$ are disjoint. 
\end{proof}

Let $\Sigma$ be a finite set of polytopes in $N_\R$.
To deal with $\Sigma$-transversality, we look at the following conditions for a polytopal decomposition $\Ccal$ of $N_\R$ assuming again that $\Lambda$ acts by translation on $\Ccal$. We fix a system of representatives $\Ncal$ for $\Ccal_n$ with respect to this $\Lambda$-action.

\begin{art} \label{transversality conditions}
	Let $\sigma \in \Sigma$ and let $\Delta_0,\dots,\Delta_p$ be pairwise different polytopes in $\Ncal$. 
	We pick linearly independent $m_i \in M_\R$ for $i$ in a finite set $I_\sigma$ and $c_i \in \R$ such that $\A_\sigma$ is given by the intersection of the affine hyperplanes 
	\begin{equation} \label{hyperplane for sigma}
	\langle m_i ,\cdot \rangle = c_i
	\end{equation}
	in $N_\R$ with $i$ ranging over $I_\sigma$.

	For  $(m_\Delta,c_\Delta)_{\Delta \in \Ncal}$, we impose the following two conditions:
	\begin{enumerate} 
		\item \label{first imposed condition}
		If $\#(I_\sigma)+p \leq n$, we require that the vectors
		\begin{equation*} %\label{first condition}
		(m_{\Delta_j}-m_{\Delta_0})_{j=1,\dots,p}, \, (m_i)_{i \in I_\sigma}
		\end{equation*}
		are linearly independent in $M_\R$. 
		%This can be expressed in terms of non-vanishing of at least one maximal subdeterminant and hence becomes true on the complement of an algebraic hypersurface in $M_\R^\Ncal$. 
		
		\item \label{second imposed condition}
		If $\#(I_\sigma)+p = n+1$, we require that the system of $(n+1)$-inhomomogeneous equations 
		\begin{align*} %\label{first part of system of equations}
		\langle m_{\Delta_j}-m_{\Delta_0}, \omega \rangle &= c_{\Delta_j}-c_{\Delta_0} \quad (j=1,\dots,p)\\
		\langle m_i , \omega \rangle &= c_i \quad (i\in I_\sigma) \label{second part of system of equations}
		\end{align*}
		has no solution in the $n$-dimensional variable $\omega \in N_\R$. 	
	\end{enumerate}
\end{art}

\begin{lem} \label{hypersurface argument}
	There is an algebraic hypersurface $H$ in $(M_\R \times \R)^\Ncal$ such that for all  $(m_\Delta,c_\Delta)_{\Delta \in \Ncal}$ in the complement of $H$ the above conditions \ref{first imposed condition} and \ref{second imposed condition} hold for all $\sigma \in \Sigma$ and all pairwise different $\Delta_0,\dots,\Delta_p \in \Ncal$ simultaneously. 	
\end{lem}

\begin{proof}
	We pick $\sigma \in \Sigma$ and pairwise different polytopes $\Delta_0, \dots, \Delta_p \in \Ncal$. Using coordinates with respect to a basis in $M$,  condition \ref{first imposed condition} is equivalent to the non-vanishing of at least one maximal subdeterminant of the matrix formed by the displayed vectors and hence becomes true on the complement of an algebraic hypersurface in $M_\R^\Ncal$. 
	Similarly, condition \ref{second imposed condition} is equivalent to the non-vanishing of the determinant of the extended $(n+1)\times(n+1)$-matrix of the system of inhomogeneous equations. We conclude 
	that the set of  $(m_\Delta,c_\Delta)_{\Delta \in \Ncal}$ fulfilling both conditions is the complement of an algebraic hypersurface in $(M_\R \times \R)^\Ncal$.
	
	Using that   $\Sigma$ and $\Ncal$ are finite sets, we conclude that there is an algebraic hypersurface $H$ in $(M_\R \times \R)^\Ncal$ such that for all  $(m_\Delta,c_\Delta)_{\Delta \in \Ncal}$ in the complement of $H$ the conditions \ref{first imposed condition} and \ref{second imposed condition} hold for all $\sigma \in \Sigma$ and all pairwise different $\Delta_0,\dots,\Delta_p \in \Ncal$ simultaneously. 
\end{proof}

Let $\Sigma$ be a finite set  of  polytopes in $N_\R$ and let $H$ be the hypersurface in $(M_\R \times \R)^\Ncal$   from Lemma \ref{hypersurface argument}.

\begin{lem} \label{transverse intersection}
	Let $f=f_{(m,c)}$ for a strictly convex piecewise linear function with respect to $\Ccal$ satisfying the cocycle rule. Then there is an open neighbourhood $U$ of $(m,c)\in (M_\R \times \R)^\Ncal$ such that for any $(m',c') \in U \setminus H$, the maximal domains of linearity form a polytopal decomposition $\Dcal(f_{(m',c')})$ of $N_\R$ which is $\Sigma$-transversal.
\end{lem}

\begin{proof}
%As before, we choose $\delta>0$ sufficiently small such that the $\delta$-centers $C(\Delta,\delta)$ are non-empty for all $\Delta \in \Ccal_n$.  
Let us choose $\delta>0$ and the open neighbourhood $U$ of $(m,c)$ as in Lemma \ref{periodic decomposition for approximations}. It is enough to show that for $(m',c') \in U \setminus H$, the polytopal decomposition $\Ccal' \coloneqq \Dcal(f_{(m',c')})$ of $N_\R$ is $\Sigma$-transversal. Since $(m',c') \not \in H$, the conditions \ref{first imposed condition} and \ref{second imposed condition} from \ref{transversality conditions} are satisfied. 
We want to use the criterion from Lemma \ref{transversality lemma}. Let $\sigma \in \Sigma$ and 
$\Delta' \in \Ccal'$. 
%We have seen in Lemma \ref{approximations and decompositions} that there is a unique $\Delta \in \Ccal_n$ such that $\Delta'$ contains the $\delta$-center $C(\Delta,\delta)$ of $\Delta$. 
For $\A_\sigma$, we use the description in \eqref{hyperplane for sigma} and hence  
\begin{equation} \label{dimension Asigma} 
\dim(\sigma)  = \dim(\A_\sigma)=n-\#(I_\sigma).
\end{equation}
On the other hand, we have
$\Delta' = \Delta_0' \cap \dots \cap \Delta_p'$ for some pairwise different $\Delta_0', \dots, \Delta_p' \in \Ccal_n'$. By the bijective correspondence between $\Ccal_n$ and $\Ccal_n'$ from Lemma \ref{periodic decomposition for approximations}, we know that for every $j=0,\dots,p$, there is a unique $\Delta_j \in \Ccal_n$ such that the $\delta$-center $C(\Delta_j,\delta)$ of $\Delta_j$ contains $\Delta_j'$.
By definition, we have $\Ccal'=\Dcal(f_{(m',c')})$ and hence Lemma \ref{periodic decomposition for approximations}\ref{presentation of f} shows
$$\Delta_j' =\{\omega \in N_\R \mid \langle m_{\Delta_j}', \omega \rangle + c_{\Delta_j}'=f_{(m',c')}(\omega)\}$$ 
We conclude that $\A_{\Delta'}$ is the set of solutions of the $p$ inhomogeneous linear equations
$$\langle m_{\Delta_j}'-m_{\Delta_0}', \omega \rangle = c_{\Delta_0}'-c_{\Delta_p}' \quad (j=1,\dots,p)$$
in $\omega \in N_\R$. 
Using Lemma \ref{transversality conditions}\ref{first imposed condition} and the underlying linear space $\L_{\Delta'}$ of $\A_{\Delta'}$, we have
\begin{equation} \label{consequences from first condition}
\dim(\Delta')=\dim(\L_{\Delta'})=n-p. 
\end{equation}
 We conclude that
\begin{equation} \label{formula for expected dimension}
D(\sigma,\Delta')= \dim(\sigma)+ \dim(\Delta')-n=n-p-\#(I_\sigma).
\end{equation}
We assume first $D(\sigma,\Delta')\geq 0$. Then Lemma \ref{transversality conditions}\ref{first imposed condition} yields that 
$$\dim(\L_\sigma \cap \L_{\Delta'})= n-p-\#(I_\sigma).$$
It follows from Remark \ref{expected dimension} that  condition \ref{first condition for transversality} in Lemma \ref{transversality lemma} is fulfilled. 

Now we assume that $D(\sigma,\Delta')< 0$. Recall that $\A_\sigma \cap \A_{\Delta'}$ is the set of solutions of the  $p+\#(I_\sigma)$ inhomogeneous linear equations in \ref{transversality conditions}\ref{second imposed condition}. Using \eqref{formula for expected dimension}, the number of equations is $>n$ and hence \ref{transversality conditions}\ref{second imposed condition} yields that $\A_\sigma \cap \A_{\Delta'} = \emptyset$. It follows that condition  \ref{second condition for transversality} in Lemma \ref{transversality lemma} is also fulfilled and hence this lemma proves $\Sigma$-transversality of $\Ccal'$.
\end{proof}

\begin{proof}[Proof of Proposition \ref{transversal pl approximation}]
By 	\cite[Proposition 8.2.6]{burgos-gubler-jell-kuennemann2}, the function $f$ is a uniform limit of piecewise $(\Q,\Gamma)$-linear functions satisfying the cocycle rule. So we may assume that the function $f$ is piecewise linear, but we do not require any rationality for $f$ at the moment. Let $\Ccal \coloneqq \Dcal(f)$ be the locally finite polytopal decomposition of $N_\R$ given by the maximal domains of linearity for $f$ as in Lemma \ref{convex piecewise linear functions}\ref{maximal domains form decomposition}. 
\walter{We have seen in Lemma \ref{convex piecewise linear functions}\ref{almost periodicity}  that $\Ccal$ is almost $\Lambda$-periodic and so
%Using the cocycle rule and that the coycles $z_\lambda$ grow quadratically in $\lambda \in \Lambda$, the domains of linearity  are indeed bounded and hence are polytopes. The cocycle rule also shows  that if $\Delta \in \Ccal$, then  $\Delta+ \lambda \in \Ccal$, but we cannot guarantee that $\Delta+ \lambda$ is disjoint from $\Delta$ for any non-zero $\lambda \in \Lambda$. However, it is clear that $\relint(\Delta)+ \lambda$ is disjoint from $\relint(\Delta)$ as otherwise they agree and an inductive argument adding successively $\lambda$ would show that $\Delta$ is unbounded. This means that  
% $(\Delta+\lambda)\cap \Delta$ is contained in the boundary of $\Delta$.
we replace $\Ccal$ by the barycentric subdivision which is a locally finite $\Lambda$-periodic simplex decomposition of $N_\R$ in the sense of Definition \ref{periodic decompositions}.}  We decrease the value of the piecewise linear function $f$ slightly in the barycenters of the faces. Using an inductive procedure starting with the barycenters of the $n$-dimensional faces and a small enough change, the resulting \walter{piecewise linear function} is strictly convex with respect to this new subdivision $\Ccal$. Of course, we have to take care during this procedure that this strictly convex $f$ still satisfies the cocycle rule which can be easily done using Lemma \ref{periodicity of piecewise linear functions}.

We have seen in Lemma \ref{convex piecewise linear functions}\ref{translation of maximal domains}   that $\Lambda$ acts on $\Ccal$ by translation. 
We fix a system of representatives $\Ncal$ of $\Ccal_n$ with respect to this $\Lambda$-action. For $\Delta \in \Ccal$, let $m_\Delta \in M_\R$ be the slope of $f|_\Delta$ and let $c_\Delta \in \R$ be the constant term. We set $(m,c)\coloneqq (m_\Delta,c_\Delta)_{\Delta \in \Ncal} \in (M_\R \times \R)^\Ncal$.
Using that $f$ is a convex piecewise linear function, it is clear from \ref{periodicity of piecewise linear functions}
 that 
$$f = f_{(m,c)} = \sup \{ m_{\Delta+\lambda} + c_{\Delta + \lambda} \mid \Delta \in \Ncal, \, \lambda \in \Lambda\}.$$

We fix $\varepsilon$ with $0<\varepsilon <1$. Our goal is to define a function  $f'\colon N_\R \to \R$ with $|f-f'| < \varepsilon$ and the desired properties.
Our strategy is to define the approximation $f'$ by picking an approximation $(m',c')$ of $(m,c)$ in $(M_\R \times \Gamma)^\Ncal$ and then setting $f' \coloneqq f_{(m',c')}$. By Lemmas \ref{convex piecewise linear functions} and \ref{approximation function}, such an $f'$ is a strictly convex piecewise linear function with respect to the polytopal decomposition $\Ccal' \coloneqq \Dcal(f')$. Moreover, we have seen there that $f'$ satisfies the cocycle rule with respect to the given cocycle $z$. 
%\begin{equation} \label{definition of approximation}
%f' \coloneqq f_{(m',c')}=\sup \{ m_{\Delta+\lambda}' + c_{\Delta + \lambda} \mid \Delta \in \Ncal, \, \lambda \in \Lambda\}.$$
%\end{equation}
In the following, we will deduce the desired properties for $f'$  assuming that the approximations $(m',c')$ of $(m,c)$ are sufficiently good. This will always mean that the imposed conditions hold in a sufficiently small neighbourhood of $(m,c)$ in the space of all approximations $(M_\R \times \Gamma)^\Ncal$. More precisely, we choose $\delta>0$ and the open neighbourhood $U$ of $(m,c)$ in $(M_\R \times \R)^\Ncal$ as in Lemma \ref{periodic decomposition for approximations}. In particular, Lemma \ref{periodic decomposition for approximations}\ref{periodicity} shows that the polytopal decomposition $\Ccal'=\Dcal(f')$ is $\Lambda$-periodic as $\Ccal=\Dcal(f)$ was assumed to be $\Lambda$-periodic.

Choosing $U$ sufficently small, the continuity in Lemma \ref{continuous functional} ensures that we have 
$$|f - f'|< \varepsilon$$
for all $(m',c') \in U$. 
We may assume that this $U$ works also for Lemma \ref{transverse intersection} and we denote the algebraic hypersurface from there again by $H$. We conclude from this result that $\Ccal'=\Dcal(f')$ is $\Sigma$-transversal for any $(m',c')\in U \setminus H$. Note that by dimensionality, the open set $U \setminus H$ is non-empty. 
By density of $M_\Q$ in $M$ and by density of $\Gamma$ in $\R$, there is $(m',c') \in (M_\Q\times \Gamma)^\Ncal \cap U$.   By Lemma \ref{approximation function}, the function $f'$ is piecewise $(\Q,\Gamma)$-linear, which in turn implies that $\Ccal'$ is a locally finite $(\Z,\Gamma)$-polytopal decomposition using Lemma \ref{convex piecewise linear functions}\ref{rationality of maximal domains}. Since $|f-f'|<\varepsilon$, this proves the proposition.
\end{proof}

\begin{rem} \label{finitely generated sigma and transversal}
We note that Proposition \ref{transversal pl approximation} also holds  for an infinite set $\Sigma$  which includes with a polytope all its faces if there is a finite set of polytopes $\Sigma'$ such that $\Sigma \subset \Sigma'+\Lambda$.  To see this we may assume that if $\Sigma'$ contains a polytope $\sigma$, then it includes all the faces of $\sigma$. Now we can apply Proposition \ref{transversal pl approximation} to the finite set $\Sigma'$ and finally we note that $\Sigma'$-transversality of the locally finite polytopal decomposition $\Ccal_k$ is equivalent to $\Sigma'+\Lambda$ transversality using that $\Ccal_k$ is $\Lambda$-translation invariant as well.

In \cite[8.1]{gubler-2007a}, the notion of \emph{$\Sigma$-generic polytopal decompositions} has been defined and it was shown that such decompositions are $\Sigma$-transversal. It is proven in the second \walter{author's thesis \cite{stadloeder}} that we can more generally assume for the approximations $f_k$ in Proposition \ref{transversal pl approximation}  that   $\frac{1}{m}\Ccal_k$ are $\Sigma$-generic polytopal decompositions for all non-zero $m \in \Z$. \walter{This will not be used in this paper.}
\end{rem}

\section{Toric metrics} \label{section tormet}

Let $K$ be an algebraically closed non-archimedean field with additive value group $\Gamma$. We consider an abelian variety $A$ over $K$. Recall from \cite[Theorem 9.5.4]{bombieri_gubler06:heights_diophant_geometry} that a rigidified line bundle $L$ on $A$ has a canonical metric $\metr_L$. If $L$ is ample, then $\metr_L$ is a continuous semipositive metric of $\Lan$ \cite[2.10]{gubler-2007b}.  

\begin{art} \label{Raynaud's uniformization theory}
	We first recall Raynaud's uniformization theory \walter{based on Raynaud's program announced in \cite{raynaud70} and worked out by Bosch and L\"utkebohmert \cite{bosch-luetkebohmert1991}, see also \cite[\S 6.5]{berkovich-book} for the formulations in the language of Berkovich spaces.} There is a unique compact subgroup $A_0$ of $A^\an$, which is an analytic subdomain, and the generic fiber of a formal group scheme $\Afrak_0$ over $\kcirc$, whose special fiber is a semiabelian variety. There is a unique formal affine torus $\T_0$ over $\kcirc$, which is a closed formal subgroup of $\Afrak_0$, and we have an exact sequence
	$$0 \longrightarrow \T_0 \longrightarrow \Afrak_0  \stackrel{q_0}{\longrightarrow} \Bfrak \longrightarrow 0$$
	of formal group schemes over $\kcirc$ where $\Bfrak$ is a formal abelian scheme, i.e.~$\Bfrak$ has good reduction. Note that $\Bfrak$ is the formal completion of an abelian scheme $\Bcal$ over $\kcirc$ \walter{(see \cite[\S 7]{bosch-luetkebohmert1991} for the argument).}
	%\walter{This follows from Riemann's period relations \cite[Theorem 6.13]{bosch-luetkebohmert1991} and RAYNAUD's THEOREM XI.1.13.} 
	Let $M$ be the character lattice of $\T_0$ and hence $\T_0=\Spf(\kcirc\{M\})$. We denote by $T = \Spec(K[M])$ the associated torus over $K$, then pushout with respect to $\T_0^\an \to \Tan$ gives the \emph{Raynaud extension}
	$$0 \longrightarrow \Tan \longrightarrow E^\an \stackrel{q}{\longrightarrow} B^\an \longrightarrow 0,$$
	which is an exact sequence of abelian analytic groups over $K$. 
	Here, the analytification of the abelian variety $B$ is the generic fiber of $\Bfrak$. The exact sequence is algebraic, but  the canonical morphism $p\colon E^{\an}\to A^\an$ is only an analytic group morphism. The kernel $\Lambda$ of $q$ is a discrete subgroup of $E(K)$ and we write $A^\an = E^\an / \Lambda$ as an identification. 
	\end{art}

\begin{art}	 \label{canonical tropicalization}
	The Raynaud uniformization $E^\an$ of $A$ comes with a \emph{canonical tropicalization map}. 
	Using that $E^\an = (A_0 \times T^{\an})/\T_0^\an$ with respect to the embedding $\T_0^\an \to A_0 \times T^\an$ given by $t \to (t,t^{-1})$, we see that the classical tropicalization map $\trop \colon T^\an \to N_\R$ for the cocharacter lattice $N=\Hom(M,\Z)$ extends to a continuous proper map $\trop\colon E^\an \to N_\R$. It is a basic fact that $\trop$ maps $\Lambda$ isomorphically onto a lattice in $N_\R$.  By passing to the quotient, we get
	$$ \tropbar \colon A^\an \longrightarrow N_\R/\trop(\Lambda),$$
	called the \emph{canonical tropicalization map of $A$}. 
	\walter{Note that the target $N_\R/\trop(\Lambda)$ is homeomorphic to the $n$-fold power of the unit circle $\rm S^1$ and that $\tropbar$ might be seen as a canonical deformation retraction of $A^\an$ onto its canonical skeleton \cite[\S 6.5]{berkovich-book}.}
\end{art}

\begin{art} \label{line bundle on the uniformization}
A line bundle $F$ on $E^\an$ descends to $A^{\an}=E^\an/\Lambda$ if and only if $F$ admits a $\Lambda$-linearization over the action of $\Lambda$ on $E^\an$. Then we have $F=p^*(\Lan)$ for the line bundle $L^\an=F/\Lambda$ on $A^\an$. Using rigidified line bundles, it is shown in \cite[Proposition 6.5]{bosch-luetkebohmert1991} that there is a line bundle $H$ on $B$, unique up to tensoring with a line bundle $E_u$ of $B$ induced from $E$ by pushout with the character $u \in M$, such that $q^*(H^\an) \cong p^*(L^\an)$ as $\Lambda$-linearized cubical line bundles. Using that as an identification and the canonical metrics on $L$ and $H$,  we note that $p^*\metr_L/q^*\metr_H$ is a continuous function on $E^\an$ which factors through the canonical tropicalization and  hence there are functions $z_\lambda \colon N_\R \to \R$ for $\lambda \in \trop(\Lambda)$ with 
\begin{equation} \label{definition cocycle}
-\log(p^*\metr_L/q^*\metr_H)(\gamma \cdot x)=-\log(p^*\metr_L/q^*\metr_H)(x)+z_{\lambda}(\trop(x))
\end{equation}
for all $x \in E^\an$ and  $\gamma \in \Lambda$ with $\lambda = \trop(\gamma)$, see \cite[4.3]{gubler-compositio}. These functions are $\trop(\Lambda)$-cocycles 
in the sense that
\begin{equation} \label{cocycle rule'}
z_{\lambda+\nu}(\omega)=z_\lambda(\omega+\nu)+z_\nu(\omega)
\end{equation}
and there is a unique symmetric bilinear form $b$ on $N_\R$ such that
\begin{equation} \label{symmetric bilinear form'}
z_\lambda(\omega)=z_\lambda(0)+b(\lambda,\omega)
\end{equation}
for all $\lambda \in \trop(\Lambda)$ and $\omega \in N_\R$. 
It follows that $z_\lambda$ is a quadratic function with associated bilinear form $b$. Using the polarization induced by $L$, we have seen
\begin{equation} \label{linear rationality property'}
b(\lambda,\cdot)\in M=\Hom(N,\Z)
\end{equation}
in \cite[Remarks 7.1.2, 8.1.3]{burgos-gubler-jell-kuennemann1}. Since $\metr_H$ is a model metric, we deduce from \eqref{definition cocycle} 
\begin{equation} \label{quadratic rationality property'}
z_\lambda(0) \in \Gamma.
\end{equation}
The line bundle $L$ is ample if and only if $H$ is ample and $b$ is positive definite, see \cite[Theorem 6.13]{bosch-luetkebohmert1991}.	We conclude that the  assumptions in \ref{cocycle} are satisfied.
\end{art}

\begin{art} \label{corresponding function}
 Now we consider any line bundle $F$ on $E^\an$ with $F=q^*(H^\an)$ for a rigidified line bundle $H$ on $B$. Using the canonical metric $\metr_H$ of $H$, there is a bijective correspondence between continuous metrics $\metr$ on $F$ and continuous real functions on $E^\an$ given by 
 $$\metr \longmapsto - \log (\metr/q^*\metr_H).$$
\end{art}

\begin{definition} \label{toric metrics on uniformization}
A continuous metric $\metr$ on $F$ is called \emph{toric} if the corresponding function factors through the canonical tropicalization of $E$.
\end{definition}

\begin{rem} \label{bijective correspondence for toric metrics}
Using \ref{corresponding function}, we get a bijective correspondence between continuous toric metrics $\metr$ on $F$ and continuous functions $f\colon N_\R \to \R$ given by
$$f \longmapsto \metr_f \coloneqq e^{-f \circ \trop} \cdot q^*\metr_H.$$
\end{rem}

\begin{rem} \label{descend of toric metrics}
As in \ref{line bundle on the uniformization}, we assume that $F=q^*(H)=p^*(L^\an)$  leading to the cocycle $(z_\lambda)_{\lambda \in \Lambda}$. For a continuous toric metric $\metr_f$ on $F$, 
it is clear that $\metr_f=p^*\metr$ for a continuous metric $\metr$ of $L$ if and only if $f$ satisfies the \emph{cocycle rule}
\begin{equation} \label{cocycle rule}
f(\omega + \lambda)= f(\omega)+ z_\lambda(\omega)
\end{equation}
for all $\omega \in N_\R$ and $\lambda \in \trop(\Lambda)$. 
\end{rem}

In the following, we consider a rigidified line bundle $L$ on $A$ with canonical metric $\metr_L$. 

\begin{definition} \label{toric metrics on A}
We call a continuous metric $\metr$ on $\Lan$  \emph{toric} if the function $-\log(\metr/\metr_L)$ is $A_0$-invariant. 
\end{definition}

\begin{prop} \label{bijective correspondence for toric metrics on A}
Let $F=p^*(\Lan)=q^*(H^\an)$ with cocycle $(z_\lambda)$ as in \ref{line bundle on the uniformization}. Then there is a bijective correspondence between continuous toric metrics $\metr$ on $\Lan$ and continuous functions $f$ on $N_\R$ satisfying the cocycle rule  \eqref{cocycle rule}, where the function $f_\metr$ associated to $\metr$ is characterized by  
$$p^*\metr=e^{-f_\metr \circ \trop} \cdot q^*\metr_H.$$
\end{prop}

\begin{proof}
This follows from Remark \ref{descend of toric metrics}.
\end{proof}

\begin{thm} \label{semipositive toric metrics}
%Under the assumptions in Proposition \ref{bijective correspondence for toric metrics on A}, if  $f_\metr$ is convex, then the continuous metric $\metr$ on $\Lan$ is semipositive. The converse holds for $L$  ample.
Under the assumptions in Proposition \ref{bijective correspondence for toric metrics on A}, assume that $L$ is ample. Then the function  $f_\metr$ is convex if and only if the continuous metric $\metr$ on $\Lan$ is semipositive.	
\end{thm}

\begin{proof}
	We first assume that  $f \coloneqq f_\metr$ is convex. Then semipositivity of $\metr$ is just a reformulation of \cite[Proposition 8.3.1]{burgos-gubler-jell-kuennemann2}.

Conversely, assume that $\metr$ is semipositive. Using currents and forms on Berkovich spaces introduced by Chambert-Loir and Ducros \cite{chambert-loir-ducros}, it is shown in \cite[Theorem 1.3]{gubler_rabinoff_jell:harmonic_trop} that the first Chern current $c_1(L,\metr)$ is positive. This means that the current evaluated at the pull-back of a compactly supported smooth positive Lagerberg form with respect to a smooth tropicalization map is non-negative. The canonical tropicalization map of $A$ is (locally) not necessarily a smooth tropicalization map \cite[\S 17]{gubler_rabinoff_jell:harmonic_trop}, but it is a harmonic tropicalization map, see \cite[Proposition 16.2]{gubler_rabinoff_jell:harmonic_trop}. We show in the appendix that the above fact also holds for pull-backs with respect to harmonic  tropicalization maps. We conclude that $d'd''[f]$ is a positive current on $N_\R$. By \cite[Proposition 2.5]{lagerberg-2012}, this is equivalent for $f$ to be convex.
\end{proof}

Next, we will see that if $f_\metr$ is a piecewise linear function, then $\metr$ is a model metric given by an explicit construction due to Mumford.

\begin{art} \label{Mumford models}
Let $\Ccal$ be a locally finite $(\Z,\Gamma)$-polytopal decomposition of $N_\R$ for the cocharacter lattice $N$ from the Raynaud extension. Then there is an \emph{associated Mumford model $\Ecal$ of $E$}. This is a scheme locally of finite type over $\kcirc$ with generic fiber $E$ and reduced special fiber $\Ecal_s$. In this context, it is often more convenient to work with formal $\kcirc$-models and we denote the formal completion of $\Ecal$ along the special fiber by $\Efrak$. We refer to \cite[\S 4]{gubler-compositio} and \cite[8.2.2]{burgos-gubler-jell-kuennemann2} for  the construction and the following properties.	

There is a bijective correspondence between the irreducible components $Y$ of $\Ecal_s=\Efrak_s$ and the vertices $\omega$ of $\Ccal$ given by the facts that the generic point of $Y$ has a unique preimage $\xi$ in $E^\an$ with respect to the reduction map and that $\trop(\xi)$ is a vertex $\omega$ of $\Ccal$. 

Let $\phi \colon N_\R \to \R$ be a piecewise $(\Z,\Gamma)$-linear function. Then $\phi$ determines a line bundle $\Ocal_\Efrak(\phi)$ on $\Efrak$ which is a formal model of $\Ocal_E$ determined by
\begin{equation}
-\log \|1\|_{\Ocal_{\Efrak}(\phi)}= \phi \circ \trop
\end{equation}
on $E^\an$ where $\metr_{\Ocal_{\Efrak}(\phi)}$ is the associated formal metric.	
\end{art}

\begin{art} \label{descend of Mumford models}
Let $L$ be a line bundle on $A$ with $F=p^*(\Lan)=q^*(H^\an)$ and cocycle $(z_\lambda)$ as in \ref{line bundle on the uniformization}. We assume that the locally finite $(\Z,\Gamma)$-polytopal decomposition 
$\Ccal$	is $\trop(\Lambda)$-periodic in the sense of \ref{periodic decompositions}. Then $\Afrak \coloneqq \Efrak/\Lambda$ is a formal $\kcirc$-model for $A$ called the \emph{formal Mumford model associated to $\overline \Ccal$}, where $\overline \Ccal$ is the polytopal decomposition $\Ccal/\trop(\Lambda)$ of $\tropbar(A^\an)=N_\R/\trop(\Lambda)$. 

 Recall \walter{from \cite[Theorem 6.7]{bosch-luetkebohmert1991}} that $H$ has a model $\Hcal$ on $\Bcal$, unique up to isomorphism. We denote by $\Hfrak$ the associated formal model on $\Bfrak$. The morphism $q\colon E \to B$ extends uniquely to a morphism $\Ecal \to \Bcal$ which we also denote by $q$. 
 For a $\trop(\Lambda)$-periodic function $\phi\colon N_\R \to \R$, let
\begin{equation} \label{H twist}
\Hfrak(f)  \coloneqq  q^*\Hfrak \otimes \Ocal_\Efrak(\phi), 
\end{equation}
where $f(\lambda) \coloneqq \phi(\lambda)+ z_\lambda(0)$ for $\lambda \in N_\R$. 
Note that we have defined $z_\lambda(0)$ only for $\lambda \in \trop(\Lambda)$, but we have seen in \ref{line bundle on the uniformization} that $z_\lambda(0)$ is a quadratic function in $\lambda$ and hence extends uniquely to a quadratic function on $N_\R$. 
The $\Lambda$-periodicity of $\phi$ is equivalent to the cocycle rule for $f$.  This yields that $\Lfrak(f) \coloneqq \Hfrak(f)/\Lambda$ is a line bundle on the formal Mumford model $\Afrak=\Efrak/\Lambda$ such that $\Lfrak(f)$ is a formal model of $L$. By Proposition \ref{toric metrics on A}, the formal metric $\metr_{\Lfrak(f)}$ is the toric metric of $L$ associated to $f$.
\end{art}

The following result will be crucial for computing toric Monge--Amp\`ere measures.

\begin{prop} \label{transversal approximation of semipositive toric metrics}
Let $\metr$ be a semipositive toric metric on an ample line bundle $L$ of $A$ and let $\Sigma$ be a finite set of polytopes in $N_\R$ including all its faces. Then $\metr$ is the uniform limit of semipositive model metrics $\metr_k$ with the following properties:
\begin{enumerate}
	\item For any $k$, there is a non-zero $m_k \in \N$ such that $\metr^{\otimes m_k}$ is the formal metric associated to a formal model $\Lfrak_k$ of $L^{\otimes m_k}$ on a formal Mumford model $\Afrak_k$ of $A$.
	\item The Mumford models $\Afrak_k$ are associated to locally finite $\trop(\Lambda)$-periodic $(\Z,\Gamma)$-polytopal decompositions $\Ccal_k$ of $N_\R$ which are $\Sigma$-transversal (see Definition \ref{transversal decomposition}).
	\item For any $k$, there is a piecewise $(\Z,\Gamma)$-linear strictly convex function $g_k$ with respect to $\Ccal_k$ satisfying the cocycle rule such that $\Lfrak_k = \Lfrak(g_k)$ by the construction in \ref{descend of Mumford models}.
\end{enumerate}
\end{prop}

\begin{proof}
We have seen in \ref{line bundle on the uniformization} that there is an ample line bundle $H$ with $p^*(\Lan)=q^*(H^\an)$ for an ample line bundle $H$ on $B$ leading to the cocycle $(z_\lambda)$. By Proposition \ref{toric metrics on A}, the toric metric $\metr$ corresponds to a continuous function
$f\colon N_\R \to \R$  satisfying the cocycle rule.  
By Theorem \ref{semipositive toric metrics}, the function $f$ is convex. We have seen in \ref{line bundle on the uniformization} that the cocycle $(z_\lambda)$ satisfies the assumptions required in \ref{cocycle} and hence we may apply Proposition \ref{transversal pl approximation}. We conclude that $f$ is the uniform limit of piecewise $(\Q,\Gamma)$-linear functions $f_k$ satisfying the cocycle rule. Moreover, we may assume that every $f_k$ is \walter{a strictly convex piecewise linear function} with respect to a locally finite $\Lambda$-periodic $\Sigma$-transversal polytopal decomposition $\Ccal_k$ of $N_\R$. By Proposition \ref{toric metrics on A} and Theorem \ref{semipositive toric metrics}, the function $f_k$ corresponds to a semipositive toric metric $\metr_k$ of $L$. Since $f_k$ satisfies the cocycle rule, we deduce from \eqref{symmetric bilinear form'} and \eqref{linear rationality property'} that there is a non-zero $m_k \in \N$ such that $g_k \coloneqq m_kf_k$ is piecewise $(\Z,\Gamma)$-linear. By   \ref{descend of Mumford models}, we get that $\metr_k^{\otimes m_k}$ is the formal metric  induced by the line bundle $\Lfrak(g_k)$ on the formal Mumford model $\Afrak_k$ associated to $\Ccal_k$.
\end{proof}
 
\begin{rem} \label{algebraicity of the Mumford model}
	Due to the analytic nature of the quotient, formal Mumford models $\Afrak=\Efrak/\Lambda$ of $A$ as in \ref{descend of Mumford models} are not necessarily algebraic. 
	But in Proposition \ref{transversal approximation of semipositive toric metrics}, we may also assume that every formal Mumford model $\Afrak_k$ is the formal completion of an algebraic model $\Acal_k$ of $A$ and that a positive tensor power of the model metric $\metr_k$ is induced by an ample model on  $\Acal_k$. Using \cite[Remark 8.2.7]{burgos-gubler-jell-kuennemann2}, this follows from strict convexity of the $f_k$. 
\end{rem}

\section{Strictly polystable alterations} \label{section strictpolyalt}

In this section, we recall strictly polystable alterations and their refinements obtained from polytopal decompositions of the skeletons. Let $K$ be an algebraically closed non-archimedean field $K$ with non-trivial additive valuation $v$ and value group $\Gamma \coloneqq v(K^\times)$ as a  subgroup of $\R$.   

We will study strictly polystable alterations for a closed subvariety $X$ of an abelian variety $A$ over a non-trivially valued algebraically closed non-archimedean field $K$ and we will relate it to Mumford models of $A$. 
At the end, we give a  degree formula which is rather technical, but will be crucial for computing the Monge--Amp\`ere measure of toric metrics in the next section. 
The material covered generalizes \cite[\S 5]{gubler-compositio} from strictly semistable to strictly polystable alterations; the arguments remain the same.

\begin{art} \label{stratification}
	Let $Y$ be any reduced  scheme locally of finite type over a field. Then $Y$ has a canonical \emph{stratification}. 	The strata of codimension $0$ are the irreducible components of the normality locus of $Y$, the strata of codimension $1$ are the irreducible components of the normality locus of the complement of the previous normality locus and so on, see \cite[\S 2]{berkovich-1999}. The strata are partially ordered by inclusion of their closures.
\end{art}

\begin{art} \label{toric schemes}
We recall the notion of toric schemes. Let $T=\mathbb G_{\rm m}^r$ be a split torus over $K$ with toric coordinates $x_1,\dots, x_r$ leading to the classical tropicalization map
$$\trop \colon 	(\mathbb G_{\rm m}^r)^\an \longrightarrow \R^r, 
\quad p \longmapsto (v(p_1),\dots, v(p_r)).$$
For a $(\Z,\Gamma)$-polytope $\Delta$ of $\R^r$, there is an \emph{associated toric  formal scheme} $\Ufrak_\Delta =\Spf(A_\Delta)$ over $\kcirc$ given by
$$A_\Delta \coloneqq \left \{ \sum_{m \in \Z^r} a_m x_1^{m_1} \dots x_r^{m_r} \mid \text{$\lim_{|m| \to \infty} v(a_m) + m \cdot \omega=\infty$ for all $\omega  \in \Delta$} \right\}$$
where $m \cdot \omega$ is the standard inner product on $\R^r$ and $|m|=m_1+\dots+m_r$. More generally, for any $(\Z,\Gamma)$-polytopal decomposition $\Dcal$ of $\Delta$, we get an associated toric formal scheme $\Ufrak_\Dcal$ over $\kcirc$ with open subsets $\Ufrak_\sigma$ for $\sigma
 \in \Dcal$ by gluing. These are admissible formal schemes with generic fiber $\trop^{-1}(\Delta)$ and reduced special fiber. We refer to \cite[\S 4]{gubler-2007a} for more details and to \cite{gubler-guide} for an algebraic description of these toric schemes.

Note that $\Tan$ has a \emph{canonical skeleton} $\Sk(T)$   given by the weighted Gauss norms and a canonical  retraction map $\tau_T\colon \Tan \to \Sk(T)$ such that $\trop\circ \tau_T=\trop$  and such that the tropicalization map restricts to a homeomorphism from $\Sk(\Ufrak_\Delta)$ onto $\R^r$, see \cite[\S 6.3]{berkovich-book}. Then we define $\Sk(\Ufrak_\Delta) \coloneqq \Sk(T)\cap \trop^{-1}(\Delta)$.
\end{art}

\begin{art} \label{strictly polystable schemes}
A \emph{non-degenerate strictly polystable formal scheme} $\Xfrak'$ over $\kcirc$ is an admissible formal scheme with reduced special fiber defined as follows. 
The formal scheme $\Xfrak'$ is covered by open affine formal schemes $\Ufrak'$ with etale morphisms $\psi\colon \Ufrak' \to \Ufrak_\Delta$ to an affine toric formal scheme associated to a $(\Z,\Gamma)$-standard polysimplex $\Delta$ in $\R^r$. Here, the number $r$ might depend on $\Ufrak'$ and a standard polysimplex is the product of standard simplices $\Delta_j$ in $\R^{r_j}$ with $r=\sum r_j$ of the form $\Delta_j=\{\omega \in [0,1]^{r_j} \mid \omega_1+\dots+\omega_{r_j} \leq \gamma_j \}$ for some $\gamma_j \in \Gamma_{\geq 0}$. 
If $\Ufrak_s'$ has a unique minimal stratum which maps to the minimal stratum of the special fiber of $\Ufrak_\Delta$, then we call $(\Ufrak',\psi)$ a \emph{building block} of $\Xfrak'$. By shrinking the above covering, we deduce easily that every non-degenerate strictly polystable formal scheme is covered by building blocks.
We refer to \cite[\S 1]{berkovich-1999}	for more details.
\end{art}

\begin{art} \label{piecewise linear structure}
Berkovich has shown that for a strictly polystable formal scheme $\Xfrak'$ over $\kcirc$, there is a \emph{skeleton} $\Sk(\Xfrak')$, given as a closed subset of $\Xfrak_\eta'$, and a canonical 
\emph{retraction map} $\tau\colon \Xfrak_\eta' \to \Sk(\Xfrak')$ which is a proper strong deformation retraction of $\Xfrak_\eta'$, see \cite[Theorem 5.2]{berkovich-1999}.

In fact, the skeleton is constructed from the building blocks $\psi\colon \Ufrak' \to \Ufrak_\Delta$ using $\Sk(\Xfrak') \cap \Ufrak_\eta'=\Sk(\Ufrak')$ and $\Sk(\Ufrak')=\psi^{-1}(\Sk(\Ufrak_\Delta))$. Since $\psi$ restricts to a homeomorphism from $\Sk(\Ufrak')$ onto  $\Sk(\Ufrak_\Delta)$ and the latter is mapped by $\trop$ homeomorphically onto the polysimplex $\Delta$, we can endow the skeleton $\Sk(\Xfrak')$ with a piecewise $(\Z,\Gamma)$-linear structure coming with \emph{canonical faces} $\Sk(\Ufrak')$ related to the building blocks such that the canonical face $\Sk(\Ufrak')$ is isomorphic to the polysimplex $\Delta$ via $\trop \circ \psi$. 
We refer to  \cite[\S 5]{berkovich-1999} for details.

The canonical faces are in bijective  correspondence to the strata of $\Xfrak_s'$. The canonical face $\Delta_S$ of $\Sk(\Xfrak')$ corresponding to a stratum $S$ of $\Sk(\Xfrak')$ is determined by  
$\relint(\Delta_S)= \tau(\red^{-1}(S))$. This stratum-face correspondence is order reversing and hence the irreducible components of $\Xfrak_s'$ correspond to the vertices of $\Sk(\Xfrak')$.
\end{art}

\walter{In the following, we define a \emph{polytope in $\Sk(\Xfrak')$} as a (convex) polytope contained in a canonical face of $\Sk(\Xfrak')$ identifying the latter with a polysimplex $\Delta$ as above.}

\begin{definition} \label{polytopal of skeleton}
	Let $\Xfrak'$ be a strictly polystable formal scheme over $\kcirc$ with skeleton $\Sk(\Xfrak')$. Then a \emph{polytopal subdivision} of $\Sk(\Xfrak')$ is a finite set $\Dcal$ of polytopes \walter{in $\Sk(\Xfrak')$}
	%, each contained in a canonical face of $\Sk(\Xfrak')$, 
	such that for every stratum $S$ the set $\Dcal_S \coloneqq \{\Delta \in \Dcal \mid \Delta \subset \Delta_S \}$ is a polytopal decomposition of $\Delta_S$.
\end{definition}

\begin{art} \label{refinement of polystable}
Let $\Xfrak'$ be a strictly polystable formal scheme over $\kcirc$ with generic fiber $X'$ and let  $\Dcal$ be a $(\Z,\Gamma)$-polytopal subdivision of $\Sk(\Xfrak')$ as above. Then there is an associated formal $\kcirc$-model $\Xfrak''$ of $X'$ with reduced special fiber and with a morphism $\iota \colon \Xfrak'' \to \Xfrak'$ extending the identity on $X'$. Locally, over a building block $\Ufrak'$ with etale morphism $\psi\colon \Ufrak'\to \Ufrak_\Delta$, the preimage $\Ufrak''$ of $\Ufrak'$ with respect to $\iota'$ is given by the cartesian diagram
\begin{equation}  \label{diagram of can morphisms}
\begin{CD} 
\Ufrak'' @>\psi'>> \Ufrak_{\Dcal_S}\\
@VV{\iota'}V    @VV{\iota}V\\
\Ufrak' @>{\psi}>> \Ufrak_\Delta
\end{CD}
\end{equation}
of formal schemes over $\kcirc$ and in general we obtain $\Xfrak''$ and $\iota'$ by gluing. Here, we used the induced polytopal decomposition $\Dcal_S$ of $\Delta = \Delta_S$  and 
 the canonical morphism $\iota$ of the toric formal $\kcirc$-models from \ref{toric schemes}. We refer to \cite[\S 5.6]{gubler-compositio} for more details in the strictly semistable case and to \cite[Remark 5.19]{gubler-compositio} for the generalization to the polystable case.
\end{art}

\begin{art} \label{properties of refinement}
We will now describe the crucial properties of the above formal $\kcirc$-model $\Xfrak''$. We refer to \cite[Proposition 5.7, Corollary 5.8]{gubler-compositio} for the arguments which generalize to our polystable setting \cite[Remark 5.19]{gubler-compositio}. There is again a bijective order-reversing  correspondence between the strata $R$ of $\Xfrak''$ and the  faces $\sigma$ of $\Dcal$ given by 
\begin{equation} \label{stratum-face correspondence}
R= \red\left(\tau^{-1}(\relint(\sigma))\right), \quad \relint(\sigma)= \trop\left(\red^{-1}(Y)\right),
\end{equation}	
where $Y$ is any non-empty subset of $R$. We have $\dim(\sigma)=\codim(R,\Xfrak_s'')$ and hence the irreducible components $Y$ of $\Xfrak_s''$ are in bijective correspondence to the vertices $\xi$ of $\Dcal$. The vertex corresponding to $Y$ is the unique point $\xi$ of $X'$ with $\red(\xi)$ being the generic point of $Y$.

Let $R$ be a stratum of $\Xfrak_s''$ with corresponding face $\Delta \in \Dcal$. Then $\relint(\Delta)$ is contained in the relative interior of a unique canonical face $\Delta_S$ of $\Sk(\Xfrak')$ corresponding to a stratum $S$ of $\Xfrak_s'$. Then $R$ is a fiber bundle over $S$ via $\iota'$ with the fiber being a torus of rank $\codim(\Delta,\Delta_S)$ and hence $R$ is smooth. The closure of $R$ is the union of the strata corresponding to the faces $\sigma \in \Dcal$ with $\Delta \subset \sigma$. 
\end{art}

\begin{definition} \label{polystable alteration}
Let $X$ be a proper variety over $K$ with formal $\kcirc$-model $\Xfrak$ over $\kcirc$. Then a \emph{strictly polystable alteration} is a generically finite proper morphism $X' \to X$ from a smooth variety $X'$ over $K$ which extends to a morphism $\varphi\colon \Xfrak' \to \Xfrak$ for a non-degenerate strictly polystable formal $\kcirc$-model $\Xfrak'$ of $X'$. 
\end{definition}

\begin{rem} \label{existence of polystable alteration}
It has been shown in \cite[Theorem 5.2.19]{adiprasito_etal} that a strictly polystable alteration always exists, at least when $\Xfrak$ is algebraic. By \cite[Lemma 2.4]{gubler-martin}, any formal $\kcirc$-model is dominated by the formal completion of an algebraic $\kcirc$-model. 
\end{rem}
 
\begin{art} \label{setup}
Now we fix the following setup. Let $X$ be a closed subvariety of the abelian variety $A$. 
We  use the Raynaud extension $$0 \longrightarrow \Tan \longrightarrow E^\an \stackrel{q}{\longrightarrow} B^\an \longrightarrow 0$$
and the notation from the previous section. 
We choose a formal Mumford model $\Afrak_0$ of $A$ over $\kcirc$ associated to a $\trop(\Lambda)$-periodic $(\Z,\Gamma)$-polytopal decomposition $\Ccal_0$ of $N_\R$ and we denote by $\Xfrak$ the closure of $\Xan$ in $\Afrak_0$ as in \cite[Proposition 3.3]{gubler-crelle}. 
\walter{It is called closure as it is similar to the construction of the schematic closure of $X$ in an algebraic model of $A$ over $\kcirc$.}
 We assume that there is a strictly polystable alteration $\varphi_0\colon \Xfrak' \to \Xfrak$. We denote the generic fiber of $\varphi_0$ by $f\colon \Xfrak_\eta' \to \Xan$. Using Remarks \ref{algebraicity of the Mumford model} and \ref{existence of polystable alteration}, for any closed subvariety $X$ of $A$ such a Mumford model $\Afrak_0$ with a strictly polystable alteration for $\Xfrak$ exists.
\end{art}

\begin{art} \label{linearization of f}
By \cite[Proposition 5.11, Remark 5.19]{gubler-compositio}, there is a unique map 
$$\overline{f}_{\rm aff}\colon \Sk(\Xfrak')\longrightarrow \tropbar(A^\an)=N_\R/\trop(\Lambda)$$ with $\overline{f}_{\rm aff} \circ \tau = \tropbar \circ f$. For every canonical face $\Delta'$ of $\Sk(\Xfrak')$, there is a unique face $\overline\Delta$ of the polytopal decomposition $\overline{\Ccal_0}$ of $N_\R/\trop(\Lambda)$ such that $ \overline{f}_{\rm aff}(\relint(\Delta'))\subset\relint(\overline\Delta)$. Moreover, the restriction of $\overline{f}_{\rm aff}$ to $\Delta'$ is a $(\Z,\Gamma)$-affine map. We denote by $f_{\rm aff}\colon \Sk(\Xfrak') \to N_\R$ a lift of $\overline{f}_{\rm aff}$ which might be multi-valued and which is unique up to $\trop(\Lambda)$-translation. Note that the restriction of $f_{\rm aff}$ to a canonical face $\Delta'$ is a single-valued affine function, unique up to $\trop(\Lambda)$-translation. \end{art}

\begin{art} \label{lifts to the abelian scheme}
Using  the uniformization  $A^\an=E^\an/\Lambda$, there is a multi-valued continuous lift $F\colon \Xfrak_\eta' \to E^\an$ of $f$ which is unique up to $\Lambda$-translation. Then $q \circ F$ extends to a multi-valued continuous morphism $G\colon \Xfrak' \to \Bfrak$ for the formal abelian scheme $\Bfrak$ over $\kcirc$ associated to $B$. To omit multi-valued morphisms, we consider a stratum $S$ of $\Xfrak_s'$. 
One can show that the restriction of $G$ to $\overline S$ is a morphism which is canonical up to $q(\Lambda)$-translation. 
\walter{This is based on the fact that $\red_{\Xfrak'}^{-1}(\overline{S})$ is contractible as its skeleton $\Delta_{S}$ is contractible and hence the restriction of $f$ to  $\red_{\Xfrak'}^{-1}(\overline{S})$ lifts to the universal cover $E^\an$ of $A^\an$.}
We refer to \cite[Remarks 5.16 and 5.19]{gubler-compositio} for details.
\end{art}

\begin{art} \label{line bundle for degree}
Let $L$ be a rigidified line bundle on $A$. Then there is a rigidified line bundle $H$ on $B$ with $p^*(\Lan)=q^*(H^\an)$ and cocycle $(z_\lambda)$ as in \ref{line bundle on the uniformization}. Let us consider a function $h\colon N_\R \to \R$ which is piecewise $(\Z,\Gamma)$-linear with respect to the $\trop(\Lambda)$-periodic $(\Z,\Gamma)$-polytopal decomposition $\Ccal_1$ of $N_\R$.  Let $\Afrak_1$ be the associated formal Mumford model of $A$, let $\Lfrak=\Lfrak(h)$ be the line bundle  on $\Afrak_1$ induced by $h$ and let $\Hfrak$ be the model of $H$ on $\Bfrak$, see \ref{descend of Mumford models}. Similarly as in \cite[5.17]{gubler-compositio}, we see that $\Sk(\Xfrak')$ has the $(\Z,\Gamma)$-polytopal subdivision
\begin{equation} \label{decompositon identity}
\Dcal= \{ \Delta_S \cap \overline{f}_{\rm aff}^{-1}(\overline\sigma) \mid \text{$S$ stratum of $\Xfrak'$ , $\overline\sigma \in \overline{\Ccal_1}$}\}
\end{equation}
such that $f\colon \Xfrak_\eta' \to A^\an$ extends to a morphism $\varphi_1\colon \Xfrak'' \to \Afrak_1$. Here, we use the formal scheme $\Xfrak''$ over $\Xfrak'$ associated to the subdivision $\Dcal$ by the construction in \ref{refinement of polystable}.

Our goal is to compute the degree of an irreducible component $Y$ of $\Xfrak_s''$ with respect to (the pull-back of) $\Lfrak$. By  \ref{properties of refinement}, $Y$ corresponds to a vertex $\xi_Y$ of $\Dcal$. 
Let $\overline \sigma$ be the unique face of $\overline{\Ccal_1}$ such that $\overline{f}_{\rm aff}(\xi_Y)$ is contained in $\relint(\overline \sigma)$.
 Since $\xi_Y$ is a vertex of the polytopal subdivison given by \eqref{decompositon identity}, we conclude that $\overline{f}_{\rm aff}$ is injective on $\Delta_S$ and that 
\begin{equation} \label{vertex intersection}
\overline{f}_{\rm aff}(\xi_Y)= \overline{f}_{\rm aff}(\Delta_S)\cap \overline\sigma.
\end{equation}
Since $\Delta_S$ is a $(\Z,\Gamma)$-polytope, the underlying linear space $\L_{\Delta_S}$ has a well-defined $\Z$-linear structure which we will use to compute Monge--Amp\`ere measures in the following result.
\end{art}	
	
\begin{prop} \label{degree formula in vertex}
Using the above notation, we assume that $h\circ f_{\rm aff}$ is convex in $\xi_Y$ and we denote by $h_{Y}$ the induced conic convex function in $\xi_Y$ on the linear space   $\L_{\Delta_S}$. If the intersection in \eqref{vertex intersection} is transversal, which means $\dim(\Delta_S)=\codim(\sigma,N_\R)$, then 
$$\deg_\Lfrak(Y)=\frac{d!}{e!} \cdot \deg_\Hfrak(\overline S) \cdot \MA(h_Y)(\{\xi_Y\})$$
where $d \coloneqq \dim(X)$, $e$ is the dimension of the stratum $S$ and the real  Monge--Amp\`ere measure on the right is computed with respect to the $\Z$-linear structure of $\L_{\Delta_S}$.
\end{prop}

\begin{proof} Note that the Monge--Amp\`ere measure $\MA(h_Y)$ is a discrete measure on $\Delta_S$ supported in the vertex $\xi_Y$ of $\Dcal$, see \ref{subsection real MA} for a description. The strictly semistable case has been proven in  \cite[Proposition 5.18]{gubler-compositio} and the arguments generalize to the case of strictly polystable alterations, see \cite[Remark 5.19]{gubler-compositio}.
\end{proof}

\section{Monge--Amp\`ere measures of toric metrics} \label{section MA tormetr}
 
In this section, we use the results from the previous sections to compute the 
Monge--Amp\`ere measures of toric metrics on a closed $d$-dimensional subvariety $X$ of an abelian variety $A$ over an algebraically closed non-archimedean field $K$ with non-trivial value group $\Gamma$.
To describe Monge--Amp\`ere measures of toric metrics on $\Xan$, we choose a formal Mumford model $\Afrak_0$ of $A$ and a strictly polystable alteration $\varphi_0\colon \Xfrak' \to \Xfrak$ for the closure $\Xfrak$ of $\Xan$ in $\Afrak_0$ as in \ref{setup}. We will first compute the Monge--Amp\`ere measures for the pull-back metrics on $\Xfrak_\eta'$ and then we will use the projection formula with respect to the generic fiber $f\colon \Xfrak_\eta' \to \Xan$ of $\varphi_0$. 
  
 \begin{art} \label{preliminaries for MA}
Recall the uniformization $A^\an=E^\an/\Lambda$ from the Raynaud extension 
$$0 \longrightarrow \Tan \longrightarrow E^\an \stackrel{q}{\longrightarrow} B^\an \longrightarrow 0.$$
Let $\Bfrak$ be the formal abelian scheme over $\kcirc$ with generic fiber $B^\an$. We fix an  ample line bundle $L$ on $A$. We have seen in \ref{line bundle on the uniformization} that $L$ has an associated ample line bundle $H$ on $B$ and we denote by $\Hfrak$ the associated formal model of $H$ on $\Bfrak$. 

By Proposition \ref{bijective correspondence for toric metrics} and denoting the cocharacter lattice of $T$ by $N$, a continuous toric metric $\metr$ on $\Lan$ corresponds to a function $f_\metr\colon N_\R \to \R$, satisfying the cocycle rule. The metric $\metr$ is semipositive if and only if $f_\metr$ is a convex function, see Theorem \ref{semipositive toric metrics}. 

For a canonical face $\Delta_S$ with associated stratum $S$ of $\Xfrak_s'$,  there is an affine map $f_{\rm aff}\colon \Delta_S \to N_\R$ which is canonical up to $\trop(\Lambda)$-translation and a  morphism $G\colon \overline S \to \Bfrak_s$ which is canonical up to $q(\Lambda)$-translation, see \ref{linearization of f} and \ref{lifts to the abelian scheme}.
 \end{art}

\begin{thm} \label{toric MA on X'}
Using the above notation, a continuous semipositive toric metric $\metr$ on the ample line bundle $L$ and an $e$-dimensional stratum $S$ of $\Xfrak_s'$, we have 
$$c_1(f^*L,f^*\metr)^{\wedge d}(\Omega)= \frac{d!}{e!} \cdot \deg_\Hfrak(\overline S) \cdot \MA(f_\metr \circ f_{\rm aff}|_{\relint(\Delta_S)})(\Omega)$$
for any Lebesgue measurable subset $\Omega$ of $\relint(\Delta_S)$ where 
$\deg_\Hfrak(\overline S)\coloneqq \deg_{G^*\Hfrak}(\overline S)$.
\end{thm}

In case of a discretely valued complete field $K$, a strictly semistable alteration $\varphi_0$ and the canonical metric for $L$, this result has been shown in \cite[Theorem 6.7]{gubler-compositio}.

\begin{proof}
	Both sides of the claim are continuous with respect to uniform convergence of the semipositive metrics and weak convergence of Radon measures, hence by Proposition \ref{transversal approximation of semipositive toric metrics} we may assume that $\metr$ is a semipositive model metric of $L$ determined on a formal Mumford model $\Afrak_1$ associated to a locally finite $\trop(\Lambda)$-periodic $(\Z,\Gamma)$-polytopal decomposition $\Ccal_1$ of $N_\R$ and that $h \coloneqq f_\metr$ is a piecewise $(\Q,\Gamma)$-linear strictly convex function with respect to $\Ccal_1$. We may even assume for a given finite set $\Sigma$ of polytopes in $N_\R$ that $\Ccal_1$ is $\Sigma$-transversal. We use this for the set $\Sigma$ consisting of the polytope $f_{\rm aff}(\Delta_S)$ and all its faces. 
	
We have also seen that there is a non-zero $m \in \N$ such that $mh$ is piecewise $(\Z,\Gamma)$-linear. If we replace $(L,\metr)$ by $(L^{\otimes m},\metr^{\otimes m})$, then both sides of the claim are multiplied by $m^d$ and hence we may assume that $h = f_\metr$ is piecewise $(\Z,\Gamma)$-linear. Then Proposition \ref{transversal approximation of semipositive toric metrics} shows that $\metr$ is the model metric associated to the model $\Lfrak \coloneqq \Lfrak(h)$ of $L$ on $\Afrak_1$.

The polytopal decomposition $\Ccal_1$ induces a formal $\kcirc$-model $\Xfrak''$ of $X$ over $\Xfrak'$ given by the polytopal subdivision $\Dcal$ of $\Sk(\Xfrak')$ from \eqref{decompositon identity} and a morphism $\varphi_1\colon \Xfrak'' \to \Afrak_1$ as in \ref{line bundle for degree}. 
By construction, the metric $f^*\metr$ is the formal metric associated to the model $\varphi_1^*(\Lfrak)$ of $f^*(L)$ on $\Xfrak''$. It follows from \ref{line bundle for degree} that the Monge--Amp\`ere measure $c_1(f^*(L),f^*\metr)$ is supported in the vertices of $\Dcal$, as the latter are the Shilov points $\xi_Y$ for the irreducible components $Y$ of $\Xfrak_s''$. 
On the other hand, we note that $h\circ f_{\rm aff}|_{\Delta_S}$ is a piecewise linear convex function with respect to $\Dcal_S=\Dcal \cap \Delta_S$ and hence 
the Monge--Amp\`ere measure $\MA(g\circ f_{\rm aff}|_{\relint(\Delta_S)})$ is a discrete measure supported in those vertices of $\Dcal$ which are contained in $\relint(\Delta_S)$, see \ref{subsection real MA}. It remains to check the claim for $\Omega$ consisting of a single vertex $\xi_Y$ of $\Dcal$. Then we may replace $h\circ f_{\rm aff}$ on the right hand side by $h_Y $ for the conic piecewise linear convex function $h_Y$ in $\xi_Y$ induced by $h\circ f_{\rm aff}$ and the claim follows from Proposition \ref{degree formula in vertex}. Note that the transversality assumption there holds as $\Ccal_1$ is $\Sigma$-transversal.	
\end{proof} 
 
%\begin{rem} \label{support and non-degenerate simplices}
In the setting of Theorem \ref{toric MA on X'}, we call the canonical face  $\Delta_S$ of the skeleton $\Sk(\Xfrak')$ \emph{non-degenerate with respect to $f$} if 
\begin{equation} \label{non-degenerate}
\dim(\overline f_{\rm aff}(\Delta_S))=\dim(\Delta_S) \quad \text{and} \quad \dim(G(S))=\dim(S).
\end{equation}
Obviously, the second condition in \eqref{non-degenerate} does not depend on the choice of $G$. We define $\Sk_{\rm nd}(\Xfrak')$ as the union of all non-degenerate canonical faces with respect to $f$.

\begin{prop} \label{support and non-degenerate simplices}
\walter{Let $\metr$ be a continuous semipositive toric metric  on the ample line bundle $L$ of $A$. Using the above notation,  the support of the Monge--Amp\`ere measure $c_1(f^*L,f^*\metr)^{\wedge d}$ is contained in  $\Sk_{\rm nd}(\Xfrak')$. This applies in particular to the canonical metric $\metr_L$ of the ample line bundle $L$ and then the above support agrees with $\Sk_{\rm nd}(\Xfrak')$.}	
\end{prop}

\begin{proof}
It follows from the  proof above and especially from the degree formula in Proposition \ref{degree formula in vertex} that the support of the Monge--Amp\`ere measure $c_1(f^*L,f^*\metr)^{\wedge d}$ is contained in $\Sk_{\rm nd}(\Xfrak')$. 
If $\metr$ is the canonical metric $\metr_L$ of $L$, then the restriction of $c_1(f^*L,f^*\metr_L)^{\wedge d}$ to the relative interior of a canonical face $\Delta_S$ of $\Sk(\Xfrak')$ is a multiple of the Lebesgue measure on $\relint(\Delta_S)$ as $f_{\metr_L} \circ f_{\rm aff}$ is a quadratic function on $\relint(\Delta_S)$. 
%given by formula (32) in \cite[Theorem 6.7]{gubler-compositio}. 
It follows from the positive definiteness of the bilinear form associated to the ample line bundle $L$ that this multiple is non-zero if and only if $\Delta_S$ is non-degenerate with respect to $f$. We conclude in this case that the support of $c_1(f^*L,f^*\metr_L)^{\wedge d}$ agrees with $\Sk_{\rm nd}(\Xfrak')$.
\end{proof} 

\begin{rem} \label{canonical subset}
We note that Theorem \ref{toric MA on X'} also yields a formula  for the Monge--Amp\`ere measure of the toric metric $\metr$ restricted  to $L|_X$ by using the projection formula 
\begin{equation} \label{projection formula}
c_1(L|_X,\metr)^{\wedge d}=f_*(c_1(f^*L,f^*\metr)^{\wedge d}).
\end{equation}
We will show in the next section that $\Xan$ has a smallest subset $S_X$ containing the supports of all these canonical measures and that $S_X$ has a canonical piecewise $\qgamma$-linear structure.
\end{rem}

\section{The canonical subset} \label{section cansubset}

As in the previous section, we consider a closed $d$-dimensional subvariety $X$ of an abelian variety $A$ over an algebraically closed non-archimedean field $K$ with non-trivial value group $\Gamma$. We will show that the supports of canonical measures on $\Xan$ give rise to a canonical subset $S_X$ of $\Xan$ endowed with a canonical piecewise $\qgamma$-linear structure.

We will start with the definition of the canonical subset of $\Xan$. Then we will recall  $\qgamma$-skeletons introduced by Ducros which will be an important tool to proof our main results at the end. 

Let $L$ be a rigidified ample line bundle on $A$ and let $\metr_L$ be the canonical metric of $L$.

\begin{definition} \label{def canonical subset}
The support of the Radon measure $c_1(L|_X,\metr_L)^{\wedge d}$ is called the \emph{canonical subset of $\Xan$} and will be denoted by $S_X$.
\end{definition}

\begin{rem} \label{canonical subset and polystable alteration}
We have seen in Remark \ref{existence of polystable alteration} that there is a Mumford model $\Afrak_0$ associated to a $\trop(\Lambda)$-periodic $(\Z,\Gamma)$-polytopal decomposition $\Ccal_0$ such that for the closure $\Xfrak$ of $X$ in $\Afrak_0$, we have a strictly polystable alteration $\varphi_0 \colon \Xfrak' \to \Xfrak$. Let $f\colon \Xfrak_\eta' \to \Xan$ be the generic fiber of $\varphi_0$. By Remark \ref{support and non-degenerate simplices}, the 	 support of $c_1(f^*L, f^*\metr_L)^{\wedge d}$ is equal to $\Sk_{\rm nd}(\Xfrak')$ and hence the projection formula \eqref{projection formula} proves
\begin{equation} \label{canonical support and non-degenerate}
S_X = f(\Sk_{\rm nd}(\Xfrak')).
\end{equation}
\end{rem}

\begin{prop} \label{properties of the canonical subset}
	The canonical subset $S_X$ does not depend on the choice of the ample line bundle $L$. Moreover, for any continuous semipositive metric $\metr$ on $\Lan$, the support of the Radon measure $c_1(L|_X,\metr)^{\wedge d}$ is contained in $S_X$.
\end{prop}

\begin{proof}
	Since $\Sk_{\rm nd}(\Xfrak')$ does not depend on the ample line bundle $L$, the first claim follows from \eqref{canonical support and non-degenerate}. By Remark \ref{support and non-degenerate simplices}, the support of $c_1(f^*L, f^*\metr)^{\wedge d}$ is contained in $\Sk(\Xfrak')_{\rm nd}$ and hence the second claim follows from the projection formula \eqref{projection formula}.
\end{proof}

\begin{prop} \label{functoriality of canonical subsets}
If $\psi\colon A \to B$ is a finite homomorphism of abelian varieties over $K$, then we have  $S_{\psi(X)} = \psi(S_X)$.
\end{prop}

\begin{proof}
Let $L$ be a rigidified ample line bundle on $A$. Note that $\psi^*(L)$ is ample and that $\psi^*\metr_L$ is the canonical metric of $\psi^*(L)$. Then the claim follows from the projection formula 
$$\psi_{X,*}(c_1(\psi_X^*L,\psi_X^*\metr_L))^{\wedge d}= \deg(\psi_X) \cdot c_1(L|_X,\metr_L)^{\wedge d}$$
for Monge--Amp\`ere measures where $\psi_X \colon X \to \psi(X)$ is given by $\psi$.
\end{proof}

We will show below that $S_X$ is a \emph{$(\Q,\Gamma)$-skeleton of $\Xan$} as defined by Ducros  \cite[4.6]{ducros12:squelettes_modeles}. We will now briefly recall these notions for $\Xan$, but they can be used more generally for any topologically separated strictly analytic space (see \cite[\S 4]{ducros12:squelettes_modeles}).

\begin{art} \label{c-skeletons}
For any strictly analytic domain $Y$ in $\Xan$ and an $m$-tuple $g=(g_1,\dots,g_m)$ of invertible analytic functions  on $Y$, we define the tropicalization map
$$\trop_g \colon Y \longrightarrow \R^m, \quad y \longmapsto (-\log|g_1(y)|, \dots, - \log|g_m(y)|).$$
A compact subset $P$ of $\Xan$ consisting of Abhyankar points is called an \emph{analytic $(\Q,\Gamma)$-polytope} if there is a strictly analytic domain $Y$ containing $P$ and $g_1,\dots,g_m \in \Ocal(Y)^\times$ such that $\trop_g$ induces a homeomorphism of $P$ onto a finite union of $\qgamma$-polytopes in $\R^m$ with the following properties for the 
induced piecewise $\qgamma$-linear structure on $P$: For any  
 strictly subdomain $Z$ of $\Xan$ and any $h \in \Ocal(Z)^\times$, we require that $P \cap Z$ is a piecewise $\qgamma$-linear subspace of $P$ and that the restriction of $-\log|h|$ to $P \cap Z$ is piecewise $\qgamma$-linear.
 
 A (locally) closed subset $S$ of $\Xan$ is called a \emph{$\qgamma$-skeleton} if the analytic $\qgamma$-polytopes contained in $S$ form an atlas for a piecewise $\qgamma$-linear structure on $S$. 
 It follows from \cite[4.1.2]{ducros12:squelettes_modeles} that the piecewise $\qgamma$-linear structure on $S$ is completely determined by the underlying  set $S$ and the analytic structure of $\Xan$. 
\end{art}

\walter{We will use the following criterion of Ducros.}

\begin{lemma} \label{Ducros criterion}
Let $Y$ be an integral strictly affinoid space over $K$ and let $S$ be a compact subset of $Y$ consisting of Abhyankar points. Then $S$ is an analytic $\qgamma$-polytope of $Y$ if  the following properties hold:	
\begin{itemize}
	\item[(i)] $\trop_g(S)$ is a piecewise $\qgamma$-linear subspace of $\R^m$ for any $g_1,\dots,g_m \in \Ocal(Y)\setminus \{0\}$;
	\item[(ii)] there are  $g_1,\dots,g_m \in \Ocal(Y)\setminus \{0\}$ such that the restriction of $\trop_g$ to $S$ is injective.
\end{itemize}	
\end{lemma}

\walter{Since $S$ consists of Abhyankar points, the analytic functions $g_j$ are nowhere zero on $S$ and hence $\trop_g(S)\subset \R^m$.}
\begin{proof}
	This is  criterion 2) in \cite[Lemma 4.4]{ducros12:squelettes_modeles}.
\end{proof}

\begin{lemma} \label{union of c-polytopes}
Let $Y$ be an integral strictly affinoid space over $K$. Then a finite union of analytic $\qgamma$-polytopes of $Y$ is an analytic $\qgamma$-polytope of $Y$.
\end{lemma}

\begin{proof} 
%We will use the criterion 2) in \cite[Lemma 4.4]{ducros12:squelettes_modeles}. To show that a compact subset $S$ of $Y$ consisting of Abhyankar points is an analytic $\qgamma$-polytope of $Y$, it is enough to show the following properties:
%\begin{itemize}
%	\item[(i)] $\trop_g(S)$ is a piecewise $\qgamma$-linear subspace of $\R^m$ for any $g_1,\dots,g_m \in \Ocal(Y)\setminus \{0\}$;
%	\item[(ii)] there are  $g_1,\dots,g_m \in \Ocal(Y)\setminus \{0\}$ such that the restriction of $\trop_g$ to $S$ is injective.
%\end{itemize}
By induction, it is enough to show for analytic $\qgamma$-polytopes $P$ and $Q$ of $Y$ that $P \cup Q$ is an analytic $\qgamma$-polytope of $Y$. 
\walter{We will use the criterion of Ducros from Lemma \ref{Ducros criterion}.}
Let $g_1,\dots,g_m \in \Ocal(Y)\setminus \{0\}$. Since $\trop_g(P)$ and $\trop_g(Q)$ are finite unions of $\qgamma$-polytopes in $\R^m$, it follows that the same is true for $\trop_g(P) \cup \trop_g(Q)$ proving (i) for $P \cup Q$. To prove (ii), we include in the list $g_1,\dots,g_m$ the functions appearing in (ii) for the analytic $\qgamma$-polytopes $P$ and $Q$. Then $\trop_g$ restricts to an injective function on $P$ and also to an injective function on $Q$. We will enlarge the list to get an injective function on $P \cup Q$. For any $x \in P$, there is at most one $y \in Q$ such that $\trop_g(x)=\trop_g(y)$. Since $Y$ is affinoid, there is an analytic function $h$ on $Y$ with $|h(x)|\neq |h(y)|$. Including $h$ in the list, we conclude that $x$ is the only point in $P \cup Q$ mapping to $\trop_g(x)$. By continuity, the same holds for all $x'$ in a neighbourhood of $x$ in $P$. By compactness of $P$, we conclude that we may add a finite number of non-zero analytic functions of $Y$ to the list $g_1,\dots,g_m$ to ensure that $\trop_g$ is injective on $P \cup Q$. This proves (ii) and hence $P \cup Q$ is an analytic $\qgamma$-polytope.
\end{proof} 

We will frequently use the following notions introduced in Section \ref{section tormet}: 
We have the uniformization $A^\an=E^\an/\Lambda$ from the Raynaud extension 
$$0 \longrightarrow \Tan \longrightarrow E^\an \stackrel{q}{\longrightarrow} B^\an \longrightarrow 0,$$
where $B^\an$ is the generic fiber of a  formal abelian scheme $\Bfrak$ over $\kcirc$. Let $N$ be the cocharacter lattice of the torus $T$ and let $\tropbar\colon A^\an \to N_\R/\trop(\Lambda)$ the canonical tropicalization.

\begin{thm} \label{piecewise linear structure of canonical subset}
The canonical subset $S_X$ of $\Xan$ is a $(\Q,\Gamma)$-skeleton  of $\Xan$ for any closed subvariety $X$ of $A$. For any strictly polystable alteration $\varphi_0\colon \Xfrak' \to \Xfrak$ as in Remark \ref{canonical subset and polystable alteration} with generic fiber $f\colon \Xfrak_\eta'\to \Xan$ and any canonical face $\Delta_S$ of $\Sk(\Xfrak')$ which is non-degenerate with respect to $f$, the morphism $f$ induces a piecewise $(\Q,\Gamma)$-linear isomorphism  $\Delta_S \to f(\Delta_S)$. 
\end{thm}

\begin{proof}
	By definition, the support of a Radon measure is closed, so $S_X$ is closed in $\Xan$ and hence compact.
	We choose  a strictly polystable alteration $\varphi_0\colon \Xfrak' \to \Xfrak$ as in Remark \ref{canonical subset and polystable alteration}. Let $f\colon \Xfrak_\eta'\to \Xan$ be the generic fiber. %Since every canonical face of the skeleton $\Sk(\Xfrak')$ is compact, we conclude that $S_X$ is compact and hence closed in $\Xan$. 
	The skeleton $\Sk(\Xfrak')$ is a $(\Q,\Gamma)$-skeleton of $\Xfrak_\eta'$ \cite[Exemple 4.8]{ducros12:squelettes_modeles} and hence it consists of Abhyankar points. Since Abhyankar points cannot be contained in a lower dimensional closed analytic subset, we conclude that  $\Sk(\Xfrak')$ is contained in the finite part of the generically finite morphism $f$. By \cite[1.4.14]{ducros18:families}, it follows that $f(\Sk(\Xfrak'))$ consists of Abhyankar points. In particular, this holds for $S_X$.

	To show that $S_X$ is a   $(\Q,\Gamma)$-skeleton  of $\Xan$, we may argue $\rm G$-locally at any $x \in \Xan$ with respect to the Grothendieck topology induced by the strictly analytic domains of $\Xan$, see \cite[Proposition 4.9]{ducros12:squelettes_modeles}. 
	If $x \not \in S_X$, then $S_X$ is empty in a neighbourhood of $x$ and the claim holds. So we may assume that $x \in S_X$. %There is a canonical face $\Delta_S$ of $\Sk(\Xfrak')$ which is non-degenerate with respect to $f$ such that $x \in f(\Delta_S)$. 
	We recall from \ref{setup} that $\Xfrak$ is the closure of $X$ in the Mumford model $\Afrak_0$ associated to a $\trop(\Lambda)$-periodic $(\Z,\Gamma)$-polytopal decomposition $\Ccal_0$ of $N_\R$. We pick a lift $\tilde x \in E^\an$ of $x$ with respect to the quotient morphism $p \colon E^\an \to A^\an = E^\an/\Lambda$. There is a unique $\Delta \in \Ccal_0$ such that $\trop(\tilde x) \in \relint(\Delta)$. We fix torus coordinates $x_1, \dots, x_n$ of $T$ giving $N_\R \cong \R^n$. It is explained in \cite[4.2]{gubler-compositio} that there is a formal affine open covering of $\Bfrak$ such that, for the generic fiber $W$ of any member of the covering, the morphism $q\colon E^\an \to B^\an$ splits over $W$ and such that for $y \in \trop^{-1}(W)$, the canonical tropicalization is given by 
	\begin{equation} \label{h-representation of trop}
	\trop(y)=\trop_h(y)=(-\log|h_1(y)|, \dots, - \log|h_n(y)|)
	\end{equation}
	where $h_j \coloneqq p_1^*(x_j)$ for the first projection $p_1$ with respect to 
	the splitting $q^{-1}(W) \cong \Tan \times W$. We pick such a $W$ with $q(\tilde x)\in W$. 
	Since $\trop(\tilde x) \in \relint(\Delta)$, we deduce that $U_{\Delta,W}\coloneqq \trop^{-1}(\Delta) \cap q^{-1}(W) \cong U_\Delta \times W$ is a strictly affinoid domain of $E^\an$ containing $\tilde x$, where $U_\Delta$ is the polytopal domain of $\Tan$ given by the preimage of $\Delta$ with respect to the classical tropicalization map $\Tan \to N_\R \cong \R^n$.

	The quotient $A^\an=E^\an/\Lambda$ and the construction of the Mumford model $\Afrak_0$ identifies $q^{-1}(W)/\Lambda$ with the generic fiber $V_{\Delta,W}$ of a formal affine open subset $\Vfrak_{\Delta,W}$ of $\Afrak_0$ such that $x \in V_{\Delta,W}$. We will view $h_1,\dots, h_n$ as invertible analytic functions on $V_{\Delta,W}$. 
	It follows from \eqref{h-representation of trop} and the definitions that $\trop_h \circ f|_{\Sk(\Xfrak')}$ is a lift of $\overline{f}_{\rm aff}|_{\Sk(\Xfrak')}$ from $\tropbar(A^\an)=N_\R/\trop(\Lambda)$ to $N_\R \cong \R^n$. If $\Delta_S$ is a canonical face of $\Sk(\Xfrak')$ which is non-degenerate with respect to $f$, then we conclude that the restriction of $\trop_h$ to $f(\Delta_S)$ is injective.

	Since $x \in S_X$, we know that $x$ is an Abhyankar point and hence $x \in X_{\rm reg}^\an$. We conclude that $\Xan$ is $\rm G$-locally integral at $x$ and hence there is an integral strictly affinoid domain $Y$ of $\Xan$ with $x \in Y \subset V_{\Delta,W}$.  
	By \cite[4.6.1]{ducros12:squelettes_modeles}, it is enough to prove that $S_X \cap Y$ is a $(\Q,\Gamma)$-analytic polytope of $Y$. 
	Note that $S_X$ is the union of $f(\Delta_S)$ with $S$ ranging over all canonical faces $\Delta_S$ of $\Sk(\Xfrak')$ which are non-degenerate with respect to $f$. Therefore Lemma \ref{union of c-polytopes} yields that it is enough to show that $f(\Delta_S)\cap Y$ is an analytic $(\Q,\Gamma)$-polytope of $Y$ for any canonical face $\Delta_S$ which is non-degenerate with respect to $f$. 
	To show this, we will use \walter{the criterion of Ducros recalled in Lemma \ref{Ducros criterion}.}
	%criterion 2) in \cite[Lemma 4.4]{ducros12:squelettes_modeles} recalled in the proof of Lemma \ref{union of c-polytopes}. 
	Let $g_1,\dots, g_m$ be non-zero analytic functions on $Y$.  Since $\Delta_S$ consists of Abhyankar points, the restriction of any $g_j$ to $f(\Delta_S)$ is invertible and hence there is a strictly affinoid neighbourhood $Z$ of $f(\Delta_S)\cap Y$ in $Y$ such that every $g_j$ is an invertible analytic function on $Z$. 
	Note that $Z' \coloneqq f^{-1}(Z)$ is a strictly analytic domain of $\Xfrak_\eta'$ and we have $\Delta_S \cap f^{-1}(Y)= \Delta_S \cap Z'$.
	The analytic function $g_j'\coloneqq g_j \circ f$ is invertible on $Z'$ for $j=1,\dots,m$. We have 
	\begin{equation} \label{trop composition}
	\trop_g(f(\Delta_S)\cap Y)=\trop_{g'}(\Delta_S \cap f^{-1}(Y))=\trop_{g'}(\Delta_S \cap Z').
	\end{equation}
	Since $\Sk(\Xfrak')$ is a $(\Q,\Gamma)$-skeleton 
	and $Z'$ is a strictly analytic domain in $\Xfrak_\eta'$, it follows from \cite[4.6.2, 4.6.3]{ducros12:squelettes_modeles} that $\Delta_S \cap Z'$ is a $\qgamma$-skeleton in $Z'$. By \cite[4.6.4]{ducros12:squelettes_modeles}, the map $\trop_{g'}$ is piecewise $\qgamma$-linear on $\Delta_S \cap Z'$ and hence we deduce from \eqref{trop composition} that $\trop_g(f(\Delta_S)\cap Y)$ is a finite union of $\qgamma$-polytopes in $\R^m$. This proves (i) of the criterion. 
	
	Now we choose for $g_1,\dots,g_m$ the restrictions of the functions $h_1,\dots, h_n$ to $Y$. We have already seen that the restriction of $\trop_h$ to $f(\Delta_S)$ is injective. We conclude that the same is true for the restriction to the subset $f(\Delta_S)\cap Y$ which proves (ii) of the criterion.
	Then the criterion yields that  $f(\Delta_S)\cap Y$ is an analytic $\qgamma$-polytope of $Y$ proving that $S_X$ is a $\qgamma$-skeleton.

Since 
$\trop_h$ is injective on $f(\Delta_S)$ for any non-degenerate canonical face $\Delta_S$ of $\Sk(\Xfrak')$, we get an induced piecewise $\qgamma$-linear isomorphism $f(\Delta_S)\to \trop_h(f(\Delta_S))$. We have seen that $\trop_h \circ f|_{\Delta_S}$ is a lift of $\overline f_{\rm aff}|_{\Delta_S}$ and hence  a $\qgamma$-linear isomorphism of $\Delta_S$ onto $\trop_h(f(\Delta_S))$. Therefore $f$ induces a piecewise $\qgamma$-linear isomorphism $\Delta_S \to f(\Delta_S)$.
\end{proof}

By Theorem \ref{piecewise linear structure of canonical subset}, the set $S_X$ has a canonical piecewise $\qgamma$-linear structure.

\begin{cor} \label{canonical subset and canonical measure}
There is a polytopal $\qgamma$-decomposition $\Sigma$ of the canonical subset $S_X$ such that for any rigidified ample line bundle $L$ on $A$ with canonical metric $\metr_L$, we have 
$$c_1(L|_X,\metr_L)^{\wedge d}= \sum_{\sigma \in \Sigma} r_\sigma \mu_\sigma$$
where $\mu_\sigma$ is a fixed choice of a Lebesgue measure on the polytope $\sigma$ and where $r_\sigma \in \R_{\geq 0}$ with $r_\sigma >0$ for all maximal $\sigma$.
\end{cor}

\begin{proof}
We choose a strictly polystable alteration $\varphi_0\colon \Xfrak' \to \Xfrak$ as in Remark \ref{canonical subset and polystable alteration} with generic fiber $f\colon \Xfrak_\eta'\to \Xan$. We have seen in Theorem \ref{piecewise linear structure of canonical subset} that $f$ restricts to a surjective piecewise $\qgamma$-linear map $\Sk(\Xfrak')_{\rm nd} \to S_X$ with finite fibers. It follows that there is a $\qgamma$-polytopal decomposition $\Sigma'$ of $\Sk(\Xfrak')_{\rm nd}$ refining the canonical face structure of $\Sk(\Xfrak')$ such that $\Sigma \coloneqq f(\Sigma')$ is a polytopal decomposition of $S_X$. Note that 
$f$ restricts to a $\qgamma$-affine isomorphism $\sigma' \to \sigma \coloneqq f(\sigma')$ for all $\sigma' \in \Sigma'$. We conclude that the  push-forward of a Lebesgue measure on $\sigma'$ is a Lebesque measure on $\sigma$. 
Using the projection formula \eqref{projection formula}, the claim follows from Remark \ref{support and non-degenerate simplices}.
\end{proof}
 
 The following tropical description of the canonical Monge--Amp\`ere measures has been shown in \cite[Theorem 1.1]{gubler-compositio} in the special case of $K$ being the completion of the algebraic closure of a field with a discrete valuation. This statement is crucially used in Yamaki's reduction theorem, which is a major contribution to the proof of the geometric Bogomolov conjecture by Xie and Yuan \cite{xie_yuan}. We show here that this tropical description is true for any algebraically closed non-trivially valued non-archimedean field $K$.
 
 \begin{thm} \label{tropical description of canonical measures}
 	The canonical tropicalization map $\tropbar \colon A^\an \to N_\R/\Lambda$ gives a surjective piecewise $\qgamma$-linear map $S_X \to \tropbar(\Xan)$ with finite fibers. 
 \end{thm}
 
 \begin{proof}
 	The proof follows the same lines as in \cite[\S 7]{gubler-compositio} and so we only give a sketch, mainly pointing out the necessary adaptions. We have seen in Theorem \ref{piecewise linear structure of canonical subset} that the canonical subset $S_X$ of $\Xan$ is a $\qgamma$-skeleton in $\Xan$, which implies that $\tropbar$ induces a piecewise $\qgamma$-linear map $S_X \to \tropbar(\Xan)$. 
 	%The main point is to show surjectivity. 
 	We use a strictly polystable alteration $\varphi_0\colon \Xfrak' \to \Xfrak$ as in Remark \ref{canonical subset and polystable alteration} with generic fiber $f\colon \Xfrak_\eta'\to \Xan$. 
 	Since $\overline f_{\rm aff}$ agrees with $\tropbar\circ f$ on $\Sk(\Xfrak')$, it follows easily from $S_X=f(\Sk(\Xfrak')_{\rm nd})$ that the piecewise linear map $S_X \to \tropbar(\Xan)$ has finite fibers. It remains to prove surjectivity.
 	 We have to show that for any  $\overline{\omega} \in \tropbar(\Xan)$ there is a canonical face $\Delta_S$ of $\Sk(\Xfrak')$ which is non-degenerate with respect to $f$ such that 
 	\begin{equation} \label{surjectivity}
 	\overline{\omega} \in \overline f_{\rm aff}(\Delta_S)= \tropbar(f(\Delta_S)).
 	\end{equation}
 	Let $d-e$ be the local dimension of $\tropbar(\Xan)$ at $\overline\omega$. Using the density of the value group $\Gamma$, we may assume that $\overline\omega \in N_\Gamma/\trop(\Lambda)$ and that $\overline\omega$ is not contained in a polytope $\overline f_{\rm aff}(\Delta_T)$ of lower dimension.

 	Our tropical dimensionality assumption at $\overline\omega$ allows us to find  a $\trop(\Lambda)$-periodic $(\Z,\Gamma)$-polytopal decomposition $\Ccal_1$ of $N_\R$ such that for the unique $\Delta\in \Ccal_1$ with $\overline\omega \in \relint(\overline\Delta)$ we have $\tropbar(\Xan)\cap \overline\Delta=\{\overline\omega\}$ and $\codim(\Delta)=d-e$. 
 	Similarly as in the proof of Theorem \ref{toric MA on X'}, we have a canonical morphism $\varphi_1\colon \Xfrak'' \to \Afrak_1$ to the Mumford model $\Afrak_1$ associated to $\Ccal_1$, where $\Xfrak''$ is the formal scheme over $\Xfrak'$ associated to the subdivision 
 	$\Dcal= \{ \Delta_S \cap \overline{f}_{\rm aff}^{-1}(\overline\sigma) \mid \text{$S$ stratum of $\Xfrak'$ , $\overline\sigma \in \overline{\Ccal_1}$}\}$. 
 	The analytic domain $\tropbar^{-1}(\overline\Delta)$ is the generic fiber of a formal open subset $\Ufrak$ of $\Afrak_1$. Let $\Xfrak_1$ be the closure of $X$ in $\Afrak_1$. 
 	The special fiber of $\Xfrak_1$ has an irreducible component $Y$ intersecting $\Ufrak$. Since the map $\varphi_1$ induces a surjective proper map $\Xfrak'' \to \Xfrak_1$,  there is an irreducible component $Y'$ of $\Xfrak_s''$ mapping onto $Y$. Let $\xi'$ be the   vertex  of $\Dcal$ corresponding to $Y'$ (see \ref{properties of refinement}), then one deduces from the choice of $\Ccal_1$ that $\overline f_{\rm aff}(\xi')=\overline\omega$. We claim that the unique canonical face $\Delta_S$ of $\Sk(\Xfrak')$ with $\xi' \in \relint(\Delta_S)$ is non-degenerate with respect to $f$ which then proves \eqref{surjectivity} and the theorem. 
 	
 	Using that $\relint(\Delta_S)$ contains a vertex of $\Dcal$, one deduces that $\overline f_{\rm aff}$ is injective on $\Delta_S$ proving the first condition for non-degeneracy with respect to $f$. 
 	Since $\overline f_{\rm aff}(\Delta_S)$ contains $\overline\omega$, we have $\dim(\Delta_S)=\dim(\overline f_{\rm aff}(\Delta_S))=d-e $, hence  the corresponding stratum $S$ is $e$-dimensional. Let $G\colon \overline S \to \Bfrak$ be the morphism from \ref{preliminaries for MA}. Then it is clear that $\dim(G(S))\leq e$ and it remains to show equality.  
 	%Let $q_1(Y)$ be the image of $Y$ with respect to the canonical morphism $q_1\colon \Afrak_1 \to \Bfrak$.  
 	It is shown in \cite[Proposition 4.8]{gubler-compositio} that the strata of the  special fiber of the formal Mumford model $\Afrak_1$ correspond bijectively to the  faces of $\overline\Ccal_1$. Using the construction of Mumford models, there is a canonical multi-valued morphism $q_1\colon \Afrak_1 \to \Bfrak$. The restriction of $q_1$ to a stratum closure becomes a single-valued morphism which is canonical up to translation. 
 	We have seen in \ref{properties of refinement} that the dense stratum of $Y'$ is a fiber bundle over $S$ which can be used together with $\varphi_1(Y')=Y$ to show  that $G(\overline S)=q_1(Y)$ for a suitable choise of the morphism $q_1\colon Y \to \Bfrak_s$.
 	By \cite[Proposition 4.8]{gubler-compositio} again, if $W_{\overline\Delta}$ is the stratum corresponding to $\overline\Delta$, then $\overline W_{\overline\Delta}$ is a fiber bundle over $\Bfrak_s$ with fiber isomorphic to the $\codim(\Delta)$-dimensional toric variety given by the star of $\Delta$. 
 	Since $\xi \coloneqq f(\xi')$ is a point of $\Xan$ with reduction equal to the generic point of $Y$ and since $\tropbar(\xi)=\overline f_{\rm aff}(\xi')=\overline \omega \in \relint(\overline \Delta)$, we conclude that $Y$ is contained in $\overline W_{\overline\Delta}$. 
 	As this stratum has relative dimension $d-e$ over $\Bfrak_s$ and since $d=\dim(Y)$, we deduce that $\dim(q_1(Y))\geq e$. We conclude that 
 	$$e \geq \dim(G(S))=\dim(q_1(Y)) \geq e$$
 	proving equality everywhere and hence $\Delta_S$ is non-degenerate with respect to $f$.
 \end{proof}
 
\appendix

\section{Differential forms, currents and positivity}  \label{section app diffforms}
  
Here, we summarize  results about differential forms and currents on tropical and non-archimedean spaces.
Let $N$ be a free abelian group of rank $n$ and $N_\R$ the base extension to $\R$. In the applications below, we usually take $N=\Z^n$ and hence $N_\R=\R^n$,  which amounts to fix a basis in the lattice $N$. 

\begin{art} \label{Lagerberg forms}
	In Lagerberg's thesis \cite{lagerberg-2012}, he has introduced a bigraded differential sheaf of $\R$-algebras $A^{\bullet,\bullet}$ on $N_\R$ with differentials $d',d''$. We call the elements \emph{Lagerberg forms}. They have similar properties as the complex $(p,q)$-forms. 
	
	More generally, by restriction we get a bigraded differential sheaf $A^{\bullet,\bullet}$ of $\R$-algebras with differentials $d',d''$ on any tropical cycle $S$ of $N_\R$ and a dual notion of currents on $S$, see \cite[\S 3]{gubler-kuenne2017}. The elements of $A^{\bullet,\bullet}$ can be seen as functions on $S$ which we call \emph{smooth functions}.
\end{art}

\begin{art} \label{positive Lagerberg forms}
There is a unique involution $J$ of the sheaf $A^{\bullet,\bullet}$ which leaves the smooth functions invariant and satisfies $Jd'=d''J$. A smooth $(p,p)$-form $\alpha$ on $S$ is called \emph{positive} if 
$$\alpha =  (-1)^{\frac{p(p-1)}{2}} \sum_{j=1}^m f_j \alpha_j \wedge J \alpha_j$$ 
for smooth non-negative functions $f_j$ and smooth $(p,0)$-forms $\alpha_j$ on $S$. Again, positive forms are obtained from positive forms on $N_\R$ and the latter are studied in \cite{burgos-gubler-jell-kuennemann1}. In particular, we deduce that positive Lagerberg forms on $S$ are closed under products.

A Lagerberg current $T$ on $S$ is called of type $(p,q)$ if it acts on the compactly supported $(p,q)$-forms on $S$. A Lagerberg current $T$ on $S$ of type $(p,p)$ is called \emph{symmetric} if $TJ=(-1)^pT$. A \emph{positive Lagerberg current} is a symmetric Lagerberg current $T$ of type $(p,p)$ on $S$ such that $T(\alpha)\geq 0$ for all compactly supported smooth $(p,p)$ forms $\alpha$ on $S$.
\end{art}

In the following, we denote by $X$ a good strictly analytic  space over a non-trivially valued non-archimedean field $K$. Recall that Berkovich introduced the \emph{boundary} $\partial X$ of $X$ in \cite[\S 3.1]{berkovich-book}. We call $X$ \emph{boundary-free} if $\partial X=\emptyset$. The analytification of an algebraic variety over $K$ is always boundary-free \cite[Theorem 3.4.1]{berkovich-book}.

\begin{art} \label{tropicalizations}
Let $W$ be a compact strictly analytic domain in $X$. We call $h\colon W \to \R^n$ a \emph{smooth tropicalization map}  if all the coordinate functions are given by $h_i = \walter{-\log |f_i|}$ for invertible analytic functions $f_i$ on $W$. We call $h$ a \emph{harmonic tropicalization map} if all the $h_i$ are harmonic functions,  see \cite[\S 7]{gubler_rabinoff_jell:harmonic_trop}. Since every smooth function is harmonic, every smooth tropicalization map is harmonic.  For a harmonic tropicalization map $h$, a generalization of Berkovich of the Bieri--Groves theorem from tropical geometry shows that the tropical variety $h(W)$ is a finite union of $(\Z,\Gamma)$-polytopes in $\R^n$ of dimension at most $\dim(W)$.	
\end{art}

Chambert-Loir and Ducros \cite{chambert-loir-ducros} used smooth tropicalization maps to introduce smooth $(p,q)$-forms on Berkovich spaces. In \cite{gubler_rabinoff_jell:harmonic_trop}, the smooth tropicalization maps were replaced by harmonic tropicalization maps to obtain a larger class of weakly smooth forms with better cohomological behavior. The constructions can be summarized as follows. 

\begin{prop} \label{intro: characterization weakly smooth}
	There is a bigraded differential sheaf $A_{\rm sm}^{\bullet,\bullet}$ (resp.~ $A^{\bullet,\bullet}$) of $\R$-algebras on $X$ with an alternating product $\wedge$ and differentials $d',d''$ satisfying the following properties:
	\begin{enumerate}
		\item \label{functoriality of pull-back}
		For a morphism $f\colon X' \to X$ of good strictly analytic spaces, there is a functorial homomorphism $f^*\colon A_X \to  f_*A_{X'}$ of sheaves of bigraded differential $\R$-algebras.
		\item \label{harmonic lift}
		If  $h\colon W \to \R^n$ is a smooth (resp.~harmonic) tropicalization map on a compact strictly analytic subdomain $W$ of $X$, there is an injective homomorphism $$h^*\colon A^{\bullet,\bullet}(h(W))\to A^{\bullet,\bullet}(W)$$ of bigraded differential $\R$-algebras lifting smooth Lagerberg forms from  $h(W)$ to $W$.
		\item \label{functoriality of harmonic lift}
		Using the above notation, we have $(h\circ f)^*=f^* \circ h^*$.
		\item \label{locally given by lifts}
		For $\omega \in A(X)$, any $x \in X$ has a strictly affinoid neighbourhood $W$ with a smooth (resp.~harmonic) tropicalization map $h \colon W \to \R^n$ such that $\omega|_W=h^*(\alpha)$ for some $\alpha \in A(h(W))$.
	\end{enumerate}
	We call  $A_{\rm sm}^{\bullet,\bullet}$ (resp.~$A^{\bullet,\bullet}$) the sheaf of smooth (resp.~weakly smooth) forms. These sheaves of bigraded differential $\R$-algebras are characterized up to unique isomorphisms by \ref{functoriality of pull-back}--\ref{locally given by lifts}.
\end{prop}

\begin{art}  \label{currents}
If $X$ is also boundary-free and separated, Chambert-Loir and Ducros introduced \emph{currents of type $(p,q)$} as continuous linear functionals acting on compactly supported smooth forms of bidegree $(p,q)$ \cite[\S 4]{chambert-loir-ducros}. The analogous continuous linear functionals on the space of compactly supported weakly smooth forms are called \emph{strong currents}. We denote by $D_{p,q}^{\rm sm}$ ~(resp.~$D_{p,q}$) the sheaf of currents (resp.~strong currents) of type $(p,q)$ on $X$. 

For any smooth (resp.~weakly smooth) form $\omega$, the theory of integration for top dimensional forms \cite[\S 3]{chambert-loir-ducros} yields an associated current $[\omega]_{\rm sm}$ (resp.~strong current $[\omega]$) similarly as in complex analysis, see \cite[\S 4.3]{chambert-loir-ducros} and \cite[\S 11]{gubler_rabinoff_jell:harmonic_trop}. 
\end{art}

\begin{art} \label{positive forms and currents}
We call a smooth (resp.~weakly smooth) $(p,p)$-form on $X$ \emph{positive} if it is locally given by the pull-back of a smooth positive Lagerberg form with respect to a smooth (resp.~harmonic) tropicalization map in the sense of Proposition \ref{intro: characterization weakly smooth}\ref{locally given by lifts}. Again, there is a unique involution $J$ acting on $A$ and on its subsheaf $A_{\rm sm}$ which leaves the weakly smooth functions invariant and satisfies $d'J=Jd''$. 

Now assume that $X$ is also boundary-free and separated. By duality, we also get a Lagerberg involution $J$ acting on $D^{\rm sm}$ ~(resp.~$D$). A (strong) current $T$ of type $(p,p)$ on $X$ is called \emph{symmetric} if $TJ=(-1)^pT$. We say that a symmetric (strong) current $T$ of type $(p,p)$ is \emph{positive} if $T(\omega) \geq 0$ for all compact supported (weakly) smooth positive forms $\alpha$ of bidegree $(p,p)$ on $X$.
\end{art}                                                                                                                                                             
                                                                                                                                                                      
\begin{art} \label{first Chern current}                                                                                                                               
Still assuming $X$ boundary-free and separated, we assume that $L$ is a line bundle on $X$ endowed with a continuous metric $\metr$. Then the \emph{first Chern current} (resp.~\emph{strong first Chern current}) of $(L,\metr)$ is given as follows: Locally, we choose an open subset $U$ of $X$ which is a trivialization of $L$ over $U$. Hence there is a frame $s \in L(U)$. Then the first Chern current (resp.~strong first Chern current) of $(L,\metr)$ is given on $U$ by $d'd''[-\log\|s\|]_{\rm sm}$ (resp.~$d'd''[-\log\|s\|]$). As this does not depend on the choice of the trivialization, this defines a globally defined (strong) current, see \cite[\S 6.4]{chambert-loir-ducros}. Note that the restriction of the strong first Chern current to compactly supported smooth forms agrees with the first Chern current.
\end{art}

The following result has been shown by Chambert-Loir and Ducros \cite[Lemme 5.5.3]{chambert-loir-ducros} for  smooth tropicalization maps. Their argument generalizes to  harmonic tropicalization maps. For convenience of the reader, we will provide the proof here.

\begin{prop} \label{positive smooth forms and tropicalization}
Let $X$ be a compact good strictly analytic space over $K$ of pure dimension $d$ with a harmonic tropicalization map $h\colon X \to \R^n$. We consider a smooth function $f \colon h(X) \to \R$. Then $d'd''[f\circ h]$ is a positive strong current on $X \setminus \partial X$ if and only if the restriction of $f$ to any $d$-dimensional face of the tropical variety $h(X)$ is convex.
\end{prop}

\begin{proof}
We assume first that $d'd''[f\circ h]$ is a positive strong current on $X \setminus \partial X$. Using results of Berkovich and Ducros \cite[Theorem 3.4]{ducros12:squelettes_modeles}, the tropical variety $h(X)$ is the support of a $(\Z,\Gamma)$-polytopal complex of dimension at most $d$ such that $h(\partial X)$ is contained in faces of dimension at most $d-1$.  
Let $\Delta$ be any  $d$-dimensional face of $h(X)$. Then positivity of $d'd''[f\circ h]$ on $X \setminus \partial X$ yields that $d'd''[f]$ is a positive current on $\relint(\Delta)$ and hence $f$ is a convex function on the relative interior of $\Delta$ by \cite[Proposition 2.5]{lagerberg-2012}.  By continuity of $f$, we conclude that $f$ is convex on $\Delta$. 

To prove the converse, we assume that the restriction of $f$ to any $d$-dimensional face of $\Delta$ is convex. Let $\omega$ be a positive weakly smooth form on $X \setminus \partial X$ with compact support and of bidegree $(n-1,n-1)$.   Using the theorem of Stokes, we have to prove that 
$$\int_{X \setminus \partial X} d'd''(f\circ h) \wedge \omega \geq 0.$$
%By \cite[10.13]{gubler_rabinoff_jell:harmonic_trop}, the form $\omega$ is piecewise smooth. 
We view $\omega$ as a compactly supported weakly smooth form on $X$. Since $X$ is good, there is a family $(U_i)_{i \in I}$ of strictly affinoid subdomains of $X$ whose interiors $U_i^\circ$ cover $X$ and harmonic tropicalization maps $h_i\colon U_i \to \R^{n_i}$ such that $\omega|_{U_i}=h_i^*(\alpha_i)$ for some  smooth positive $(n-1,n-1)$-Lagerberg form $\alpha_i$ on $h_i(U_i)$. 
By \cite[Proposition 3.3.6]{chambert-loir-ducros}, there is a smooth partition of unity $(\varphi_i)_{i \in I}$ subordinated to the covering $(U_i^\circ)_{i \in I}$. 
We may assume that the tropicalization map $h_i$ is a refinement of $h$, i.e.~there is  a $(\Z,\Gamma)$-affine map $L_i\colon h_i(U_i) \to h(X)$ such that $h = L_i \circ h_i$. Then the function $f_i \coloneqq f \circ L_i$ is a smooth function on $h_i(X)$ with convex restriction to each $d$-dimensional face. We conclude that
$$\int_{X \setminus \partial X} d'd''(f\circ h)\wedge \omega = \sum_{i \in I} \int_{U_i} \varphi_i \cdot h_i^*(d'd''f_i \wedge \alpha_i).$$
The support $K_i$ of the weakly smooth form $\eta_i \coloneqq \varphi_i \cdot h_i^*(f_i d'd''\alpha_i)$ is a compact subset of $U_i^\circ \cap (X \setminus \partial X)=U_i \setminus \partial U_i$ \cite[Proposition 3.1.3]{berkovich-book}. It follows from \cite[Lemme 3.2.5]{chambert-loir-ducros} that we can apply 
\cite[Proposition 3.4.4]{chambert-loir-ducros} to this compact subset $K_i$ of $U_i \setminus \partial U_i$. We conclude that there is a strictly affinoid neighbourhood $V_i$ of $\supp(\eta_i)$ in $U_i$ and a smooth tropicalization map $F_i\colon V_i \to \R^{m_i}$ which  satifies $\varphi_i|_{V_i}=\phi_i \circ F_i$ on $V_i$ for a smooth function $\phi_i$ on $F_i(V_i)$. Replacing $F_i$ by a harmonic tropicalization map refining $h_i$, we may assume that $F_i=h_i$. We conclude that
$$\int_{X \setminus \partial X} d'd''(f\circ h)\wedge \omega = \sum_{i \in I} \int_{V_i} h_i^*(\phi_i)  \cdot h_i^*(d'd''f_i \wedge \alpha_i)= \int_{h_i(V_i)}
 d'd''f_i \wedge (\phi_i\alpha_i).$$ 
Since $\phi_i\alpha_i$ is a positive Lagerberg form on $h_i(V_i)$ and since $d'd''f_i$ is a positive Lagerberg form on each maximal face $\Delta$ of $h_i(V_i)$ as $f_i|_\Delta$ is a smooth convex function, the above integral is non-negative proving the claim.
\end{proof}

\begin{cor} \label{strong positive if equivalent to positive}
With the setup of Proposition \ref{positive smooth forms and tropicalization} and assuming $h$ is a smooth tropicalization map, we have  $d'd''[f\circ h]$ is a positive strong current on $X \setminus \partial X$ if and only  $d'd''[f\circ h]_{\rm sm}$ is a positive current on $X \setminus \partial X$.	
\end{cor}

\begin{proof}
This follows immediately from Proposition \ref{positive smooth forms and tropicalization} as the same criterion holds for positivity of the current $d'd''[f \circ h]_{\rm sm}$ on $X \setminus \partial X$ with respect to the smooth tropicalization map $h$, see 	\cite[Lemme 5.5.3]{chambert-loir-ducros}. 
\end{proof}

In the following, we assume that $X$ is a good strictly analytic boundary-free separated Berkovich space (for example the analytification of an algebraic variety). We consider a line bundle $L$ over $X$ endowed with a continuous semipositive metric $\metr$. It was shown in \cite[Theorem 1.3]{gubler_rabinoff_jell:harmonic_trop} that the first Chern current of $(L,h)$ is positive. Using the above, the same arguments show the following result:

\begin{thm} \label{strong first Chern current is positive}
Let $\metr$ be a continuous semipositive metric on $L$. Then the strong first Chern current of $(L,\metr)$  is positive.
\end{thm}

\begin{proof}
By restriction to the irreducible components of $X$, we may assume that $X$ is of pure dimension. 
By assumption, the metric $\metr$ is a uniform limit of semipositive model metrics. Since such a limit obviously preserves positivity of the strong first Chern current, we may assume that $\metr$ is a semipositive model metric. Then \cite[7.14, Proposition 7.10]{gubler_rabinoff_jell:harmonic_trop} yields that $\metr$  is locally a uniform limit of smooth metrics with positive first Chern currents. By Corollary \ref{strong positive if equivalent to positive}, the strong first Chern currents of the smooth metrics are also positive. As the claim is local, the above limit argument shows positivity of the strong first Chern current of $(L,\metr)$.
\end{proof}

\bibliographystyle{alpha}
\bibliography{bib-toric}

\end{document}